\pdfoutput=1
\RequirePackage{ifpdf}
\ifpdf 
\documentclass[pdftex]{sigma}
\else
\documentclass{sigma}
\fi

\numberwithin{equation}{section}

\newtheorem{Theorem}{Theorem}[section]
\newtheorem*{Theorem*}{Theorem}
\newtheorem{Corollary}[Theorem]{Corollary}
\newtheorem{Lemma}[Theorem]{Lemma}
\newtheorem{Proposition}[Theorem]{Proposition}

\theoremstyle{definition}
\newtheorem{Definition}[Theorem]{Definition}

\newtheorem{Example}[Theorem]{Example}
\newtheorem{Remark}[Theorem]{Remark}

\usepackage{mathrsfs}
\usepackage[all]{xy}

\begin{document}

\allowdisplaybreaks

\newcommand{\arXivNumber}{2602.19885}

\renewcommand{\thefootnote}{}

\renewcommand{\PaperNumber}{093}

\FirstPageHeading

\ShortArticleName{Algebraic Integrability and Minimality of Lie Equations}

\ArticleName{Algebraic Integrability and Minimality\\ of Lie Equations for Transitive, Finite-Dimensional,\\ Non-Commutative Lie Pseudogroups\footnote{This paper is a~contribution to the Special Issue on Interactions of Poisson Geometry, Lie Theory and Symmetry in honor of Rui Loja Fernandes for his 60th birthday. The~full collection is available at \href{https://sigma-journal.com/Fernandes.html}{https://sigma-journal.com/Fernandes.html}}}

\Author{Alejandro ARENAS TIRADO~$^{\rm a}$, David BL\'AZQUEZ SANZ~$^{\rm b}$ and Guy CASALE~$^{\rm c}$}

\AuthorNameForHeading{A.~Arenas~Tirado, D.~Bl\'azquez Sanz and G.~Casale}

\Address{$^{\rm a)}$~Corporaci\'on Universitaria Minuto de Dios - UNIMINUTO, Bello, Colombia}
\EmailD{\mail{alejandro.arenas.t@uniminuto.edu.co}, \mail{alejandro.arenas@uniminuto.edu}}

\Address{$^{\rm b)}$~Departamento de Matem\'aticas, Facultad de Ciencias, Universidad Nacional de Colombia,\\
\hphantom{$^{\rm b)}$}~Sede Medell\'in, Colombia}
\EmailD{\mail{dblazquezsa@unal.edu.co}}

\Address{$^{\rm c)}$~Universit\'e de Rennes, CNRS, IRMAR-UMR 6625, 35000 Rennes, France}
\EmailD{\mail{guy.casale@univ-rennes.fr}}
\URLaddressD{\url{https://gcasale.pages.math.cnrs.fr/}}

\ArticleDates{Received February 24, 2026, in final form September 18, 2026; Published online September 29, 2026}

\Abstract{We provide an algebraic characterization of transitive, finite-dimensional algebraic Lie pseudogroups (or $\mathcal{D}$-groupoids) that are {algebraically} integrable, that is, isogenous to the action groupoid of an algebraic group action. Our approach is based on the differential Galois theory of rational connections. Under suitable hypotheses on the Lie algebra of the $\mathcal D$-groupoid, its algebraic integrability is equivalent to the triviality of the differential Galois group of its $\mathcal D$-Lie algebra. Furthermore, we investigate the structure of highly non-integrable $\mathcal{D}$-groupoids, demonstrating that if the differential Galois group of the linear differential equation of their $\mathcal D$-Lie algebra is large enough, then they are minimal in the sense that they admit no nontrivial sub-$\mathcal{D}$-groupoids of positive dimension.}


\Classification{12H05; 58H05; 14L10; 34M15}

\renewcommand{\thefootnote}{\arabic{footnote}}
\setcounter{footnote}{0}

\section{Introduction}\label{sec1}

There is a~longstanding family of problems concerning the integration of linear infinitesimal structures into non-linear ones that dates back to the classical fundamental theorems of Lie, which established a~correspondence between Lie algebras and Lie groups~\cite{Lie1893}, the problem of the integrability of non-linear Lie equations~\cite{kumpera1972lie, munoz2000integrability,rodrigues1962first}, and the integration of Lie algebroids into Lie groupoids~\cite{crainic2003integrability}. In this work, we address the problem of the algebraic integrability of the system of PDEs defining an algebraic Lie pseudogroup (or $\mathcal{D}$-groupoid), in relation to the algebraic integrability of its linearized equations (or $\mathcal{D}$-Lie algebra). The question of algebraic integrability reduces to determining when a~$\mathcal{D}$-groupoid reduces to a~classical algebraic group action. We say that such a~$\mathcal{D}$-groupoid is algebraically integrable if it is isogenous to an algebraic group action (see Definition~\ref{def_isogeny}). Previous works, such as~\cite{bc2017}, established criteria for these kinds of reductions in specific contexts (parallelisms) and under additional hypotheses (normality of the Lie algebra). In contrast to algebraically integrable $\mathcal{D}$-groupoids, we find $\mathcal{D}$-groupoids that do not admit nontrivial sub-$\mathcal{D}$-groupoids; we call such $\mathcal{D}$-groupoids \textit{minimal} (see Definition~\ref{def_minimal}). The first proof of the minimality of a~$\mathcal{D}$-groupoid was given in~\cite{Casale_BlazquezSanz_ArenasTirado_2025}. The goal of this paper is to formalize and exploit the following principle:

\textit{Under suitable hypotheses on its infinitesimal structure, a~transitive finite-dimensional $\mathcal{D}$-groupoid behaves like an algebraic group action exactly when its associated $\mathcal D$-Lie algebra has a~finite Galois group; in the opposite direction, a~nontrivial Galois group forces minimality properties and rigidity of product actions.}

More precisely, let $\mathcal G$ be a~transitive finite-dimensional $\mathcal D$-groupoid acting on an algebraic variety $M$. Its Lie algebra \smash{${\rm Lie}(\mathcal G)$} is endowed with a~connection \smash{$\nabla^{{\rm Lie}(\mathcal G)}$}. Theorems~\ref{thm:1} and~\ref{thm:2} assert that $\mathcal G$ is isogenous to the action of an algebraic group if and only if \smash{${\rm Gal}\bigl(\nabla^{{\rm Lie}(\mathcal G)}\bigr)$} is a~finite group. These theorems provide an algebraic, transitive, and finite-dimensional analogue to the integrability theorems for Lie pseudogroups and Lie algebra equations (see~\cite{kumpera1972lie, munoz2000integrability,rodrigues1962first}) and establish an equivalence between the complete integrability of the non-linear equations defining a~pseudogroup and the corresponding linear equations of its Lie algebra. Furthermore, our findings align with the Morales--Ramis--Sim\'o theory (see~\cite{morales2007integrability} and references within), which connects the integrability of Hamiltonian dynamical systems to the differential Galois integrability of their linearized equations. Our work fits into this framework by interpreting the algebraic integrability of non-linear Lie equations through the lens of their linearized counterparts and the finiteness of their differential Galois groups.\looseness=1

In a~second part, we investigate the action of $\mathcal G$, or its cartesian product $\mathcal G^n$, on $M^n$.
Theorems~\ref{thm:6} and~\ref{thm:7} explain the scarcity of sub-$\mathcal D$-groupoids in the cartesian product of non-integrable $\mathcal D$-groupoids. There is an interesting remark by Lie~\cite[Chapter 25]{Lie1893} stating that the diagonal action on $\mathbb R^2$ of the pseudogroup of transformations in one variable cannot be {described using} differential equations; thus, the diagonal pseudogroup is not, in general, a~Lie pseudogroup.\looseness=1

An interesting related result is Corollary~\ref{corollary16},
 which states that for $\mathcal D$-groupoids with centerless Lie algebras, a~necessary and sufficient condition for their algebraic integrability is the existence of the diagonal subpseudogroup in the Cartesian product action as a~$\mathcal D$-groupoid. This corollary must be compared with~\cite[Theorem 2]{malgrange2012leon}, where B.~Malgrange gives a~characterization of algebraic groups among simply transitive $\mathcal D$-groupoids: if the diagonal action of~${\rm Lie(\mathcal G)}$ on~$M^2$ has algebraic leaves intersecting each fiber of the projections generically in one point then $\mathcal G$ is isogenous to an algebraic group action. His proof uses Bialynicki--Birula's version of strongly normal extensions~\cite{bialynicki1962galois} and seems very different from the proof given here.

\begin{Remark} A~reader familiar with Cartan geometry through the perspective of $\mathscr C^\infty$ Lie groupoids equipped with multiplicative Cartan connections, in the sense of Blaom, Crainic, Fernandes, M{\u{a}}rcu{\c{t}}, Salazar and Struchiner~\cite{blaom2006geometric, Blaom2016, Blaom2016b, crainic2015multiplicative, fernandes2023multiplicative}, will notice significant analogies. Our $\mathcal D$-groupoid approach is ``from above'' dealing with finite-dimensional submanifolds $\mathcal G\subset \operatorname{Aut}(M)$ defined by Lie equations, while the Lie groupoid approach proceeds ``from below'' starting from Lie groupoids $\mathcal G$ endowed with a~multiplicative Cartan connection~\cite[Definition~3.9]{crainic2015multiplicative}. For a~$\mathcal D$-groupoid $\mathcal G$, there is a~tautological involutive Cartan connection given by
\[
\sigma_{xy} \mapsto {\rm d}_x(j\sigma_{xy}|_{\mathcal G})(T_xM).
\]
 On the other hand, Blaom~\cite[Section~2.3]{Blaom2016b} has shown that involutive Cartan connections can be integrated to \emph{pseudoactions} that are representations of the Lie groupoid into $\operatorname{Aut}(M)$. Such a~representation is locally injective, but not injective in general, nor is its image necessarily a~submanifold of $\operatorname{Aut}(M)$ determined by a~system of PDE's. In other words, the pseudogroup generated by the pseudoaction does not need to be a~Lie pseudogroup. The comparison of these two approaches opens a~number of interesting questions.
\end{Remark}

\begin{Remark}
The theory of finite-dimensional $\mathcal D$-groupoids is closely related to the theory of finite-dimensional differential algebraic groups (D-groups). A~result of Buium~\cite[Chapter~5, Theorem 3.6]{buium1992differential} shows that connected, centerless D-groups are isomorphic, in the sense of differential algebraic geometry, to algebraic groups. Our results and examples on minimality show how this result does not generalize to the context of $\mathcal D$-groupoids.
\end{Remark}

\section{Preliminaries}

We recall the necessary background on $\mathcal{D}$-groupoids, foliations, parallelisms, and Galois theory of fields with operators, and we refer the reader to Malgrange~\cite{malgrange2012leon}, Casale and Bl\'azquez-Sanz~\cite{bc2017}, Casale~\cite{Casale2004SurLG}, Bia{\l}ynicki-Birula~\cite{bialynicki1962galois}, and~\cite{Casale_BlazquezSanz_ArenasTirado_2025} for details. Throughout this section, let $M$ be a~smooth irreducible algebraic complex variety of dimension $n$.

\subsection[D-groupoids or algebraic Lie pseudogroups]{$\boldsymbol{\mathcal{D}}$-groupoids or algebraic Lie pseudogroups}

 A~\textit{Lie pseudogroup} is a~pseudogroup (a set of local diffeomorphisms closed under composition, inversion, restriction, and gluing of compatible maps) that is defined by a~system of PDE's satisfying some regularity conditions. In the algebraic geometry framework, one can drop the regularity condition and the singular version of Lie pseudogroups will be called \textit{$ \mathcal D$-groupoids}. By definition, a~$\mathcal D$-groupoid acting on an algebraic variety $M$ is an algebraic Lie pseudogroup on some dense Zariski open subset of $M$. There are several equivalent formalizations of this idea; here we will follow the exposition given in~\cite[Section 2]{bcd2022malgrange}, which is more convenient for our applications as it makes systematic use of the principal structure of the frame bundles.\looseness=1

\subsubsection[k-frame bundles and k-jet groupoids]{$\boldsymbol{k}$-frame bundles and $\boldsymbol{k}$-jet groupoids}
Let $M$ be a smooth complex and irreducible algebraic variety of dimension $n$. A~$k$-frame $r_x$ in~$M$ is a~$k$-jet of a~local biholomorphism $\bigl(\mathbb C^n,0\bigr)\to M$ defined around $0\in \mathbb C^n$; its base point~${x\in M}$ is the image of $0$. The space of $k$-frames forms a~principal bundle $R_kM\to M$ modeled over the algebraic group $\Gamma_k$ of $k$-jets of local automorphisms of $\bigl(\mathbb C^n,0\bigr)$ that naturally acts on $R_kM$ by composition on the right side. Note that $R_kM$ is a~smooth algebraic variety of complex dimension $n\binom{n+k}{n}$, and $\Gamma_k$ is an affine algebraic group of complex dimension $n\big(\binom{n+k}{k} -1 \big)$. The field \(\mathbb{C}(R_k M)\) is referred to as the field of rational differential functions of order $\leq k$. The term \emph{differential} here expresses that the functions in $R_kM$ may depend on the derivatives of functions in $M$ with respect to the Cartesian coordinates $\varepsilon_1, \dots, \varepsilon_n$ in $\bigl(\mathbb C^n,0\bigr)$.

\begin{Example}\label{ex_1}
Let $\mathbb C^n = M$ with coordinates $x_1, \dots, x_n$. Then, a~$k$-frame is the $k$-jet $j_0^kr$ of a~local {biholomorphism} $r\colon(\mathbb C^n,0)\to M$, $r(\varepsilon) = (X_1(\varepsilon), \dots, X_n(\varepsilon))$, where the functions $X_j(\varepsilon)$ are assumed to be, without loss of generality, polynomials of {degree} $\leq k$
\[X_j(\varepsilon) = \sum_{0\leq|\alpha|\leq k}\frac{1}{\alpha!}\frac{\partial^{|\alpha|}X_j}{(\partial \varepsilon)^\alpha}(0)\varepsilon^\alpha.
\]
By defining \smash{$x_{j:\alpha}(r) = \frac{\partial^{|\alpha|}X_j}{(\partial \varepsilon)^\alpha}(0)$}
we obtain a~system of \smash{$n \binom{n+k}{k}$} coordinates \smash{$\{x_{j:\alpha}\}_{|\alpha|\leq k}$} on \smash{$R_k\mathbb C^n$}; where \smash{$x_{j:\alpha}$} represents the $\alpha$-th derivative of $x_j$ with respect to
$\varepsilon = (\varepsilon_1,\dots,\varepsilon_n)$.
By the inverse function theorem, $R_kM$ identifies with the complement in \smash{$\mathbb C^{n \binom{n+k}{k}}$} of the hypersurface defined by $\{\det\left( x_{j:\epsilon_j} \right)= 0\}$, that is, the vanishing of the Jacobian. Here $\epsilon_j$ represents the unit vector in $\mathbb N^n$ with $j$-th coordinate equal to $1$. The field of rational differential functions of order $\leq k$ is~${\mathbb C(x_{j:\alpha}\mid |\alpha|\leq k)}$.
\end{Example}

Simultaneously, we may consider the groupoid $\operatorname{Aut}_k(M)$ of $k$-jets of local biholomorphisms defined around any base point in $M$. This set has the~natural algebraic structure of a~smooth variety of dimension $n$ + $n\binom{n+k}{k}$. Each element $\sigma_{xy}\in \operatorname{Aut}_k(M)$ has a~source $x\in M$ and a~target $y\in M$, and two of them may be composed if their respective target and source points coincide, so that $\operatorname{Aut}_k(M)$ has a~natural groupoid structure, moreover, Fa\`a di Bruno's formula\footnote{Fa\`a di Bruno formula gives explicit expressions for the higher-order derivatives of the composition of functions. See, for instance,~\cite{weisstein2003faa}.} shows that $\operatorname{Aut}_k(M)$ is an algebraic groupoid (meaning in the category of schemes over $\mathbb C$).

\begin{Example}
As in Example~\ref{ex_1}, let $\mathbb C^n = M$ with coordinates $x_1, \dots, x_n$. An element of~${\operatorname{Aut}(M)}$ is $j^k_x\sigma$, the $k$-jet of a~local biholomorphism $\sigma\colon (M,x)\to M$, $x+ \varepsilon \mapsto \sigma(x+ \varepsilon)$ is also given by $\sigma(x_1+ \varepsilon_1,\dots,x_n+ \varepsilon_n) = (Y_1(\varepsilon),\dots,Y_n(\varepsilon))$, where the $Y_j(\varepsilon)$ are polynomials of degree~$\leq k$ of the coordinates
\[
Y_j(\varepsilon) = \sum_{0\leq|\alpha|\leq k}\frac{1}{\alpha!}\frac{\partial^{|\alpha|}Y_j}{(\partial x)^\alpha}(0)\varepsilon^\alpha.
\]
By defining \smash{$y_{j:\alpha}(\sigma) = \frac{\partial^{|\alpha|}Y_j}{(\partial x)^\alpha}(0)$},
we obtain a~system of $n + n\binom{n+k}{k}$ coordinates \smash{$\{x_j,y_{j:\alpha}\}_{|\alpha|\leq q}$} in~$\operatorname{Aut}(\mathbb C^n$), where $y_{j:\alpha}$ represents the $\alpha$-th derivative of $y_j$ with respect to $x = (x_1,\dots,x_n)$. The field of rational functions in $\operatorname{Aut}_k\bigl(\mathbb C^n\bigr)$ is $\mathbb C(x_j, y_{j:\alpha}\mid |\alpha|\leq k)$.
\end{Example}

However, we can arrive at the algebraic structure of $\operatorname{Aut}_k(M)$ by an alternative path, taking into account that any pair of frames $j_0^kr_x$, $j_0^kr_y$ can be composed to produce a~$k$-jet of biholomorphism $j_x^k\sigma = j_x^k\bigl(r_y \circ r_x^{-1}\bigr)$. Any $k$-jet of biholomorphism can be recovered as such a~composition in a~unique way, up to an element of $\Gamma_k$, which allows us to reconstruct $\operatorname{Aut}_k(M)$ as the quotient by the diagonal action $\operatorname{Aut}_k(M) \simeq (R_kM \times R_kM )/\Gamma_k$.

This is the groupoid of gauge isomorphisms of fibers of the principal bundle $R_kM\to M$. Each $k$-jet of biholomorphism $j_x^k\sigma$ can be interpreted as a~$\Gamma_k$-equivariant map from $(R_kM)_x$ to~$(R_kM)_{\sigma(x)}$, where the application of $j^k_x\sigma$ to a~$k$-frame $r_x$ is just its composition $j_0^k(\sigma\circ r_x)$, which yields a~$k$-frame at $y$.

Now, we may consider Zariski closed subgroupoids of $\operatorname{Aut}_k(M)$, but that notion would not be enough to allow us to define $\mathcal D$-groupoids. Instead, we will consider Zariski closed subsets that are groupoids \emph{on the generic points}.

\begin{Definition} A~rational subgroupoid of $\operatorname{Aut}_k(M)$ is a~Zariski closed subset $\mathcal G \subset \operatorname{Aut}_k(M)$ such that
\begin{enumerate}\itemsep=0pt
 \item[(a)] There is an open subset $M'\subset M$ such that $\mathcal G\vert_{M'} = \mathcal G \cap \operatorname{Aut}_k(M')\mapsto M'\times M'$ is a~smooth algebraic groupoid.
 \item[(b)] For any open subset $U\subset M$ the intersection $\mathcal G\vert_U = \mathcal G \cap \operatorname{Aut}_k(U)$ is dense in $\mathcal G$.
\end{enumerate}
\end{Definition}

There is a~fundamental relationship between rational groupoids and algebras of rational functions on $R_kM$. To each rational groupoid $\mathcal G\subset \operatorname{Aut}_k(M)$, we can assign a~field of rational differential invariants of order $\leq k$
\[{\rm Inv}(\mathcal G) = \{f \in \mathbb C(R_kM)\mid
f(r) = f(\sigma\circ r) \mbox{ for all } \sigma\in \mathcal G,\, r \in (R_kM)_{s(\sigma)} \cap {\rm dom}(f)\}.
\]
Reciprocally, for any rational differential function $f\in \mathbb C(R_kM)$, we can define its groupoid of symmetries ${\rm Sym}(f)$ as the Zariski closure of the set
\[
\{\sigma\in \operatorname{Aut}_k(M) \mid f(r) = f(\sigma\circ r) \mbox{ for all } r\in (R_kM)_{s(\sigma)}\cap {\rm dom}(f)\}
\]
and for any set of rational functions $\mathbb F\subset \mathbb C(R_kM)$,
${\rm Sym}(\mathbb F) = \bigcap_{f\in \mathbb F} {\rm Sym}(f)$.
Then, as a~particular case of~\cite[Proposition~2.18]{bcd2019differential}, we have the following statement.

\begin{Proposition}
 The assignments $\mathcal G \leadsto {\rm Inv}(\mathcal G)$ and $\mathbb F\leadsto {\rm Sym}(\mathbb F)$ are inverse bijective correspondences between the sets of rational subgroupoids of $\operatorname{Aut}_k(M)$ and $\Gamma_k$ invariant subfields of~$\mathbb C(R_kM)$ containing $\mathbb C$.
\end{Proposition}

For a~given rational subgroupoid $(s,t)\colon \mathcal G\to M\times M$, there is a~Zariski open subset~${M'\subseteq M}$ such that $\mathcal G|_{M'} = (s,t)^{-1}(M'\times M')$ is a~smooth algebraic Lie subgroupoid of $\operatorname{Aut}_k(M')$. For~${x,y\in M'}$, we denote $\mathcal G_{xy} = (s,t)^{-1}(\{(x,y)\})$, which is a~principal homogeneous space with the right action of the algebraic group $\mathcal G_{xx}$ and also with the compatible left action of $\mathcal G_{yy}$.

We also denote $\mathcal G_{x \bullet} = s^{-1}(\{x\})$, $\mathcal G_{\bullet x} = t^{-1}(\{x\})$ and $\mathcal G_{xy} = (s,t)^{-1}(\{(x,y)\})$. Then we have that $t\colon \mathcal G_{x\bullet}|_{M'}\to M'$ is a~$\mathcal G_{xx}$-principal homogeneous space where the action is given by the composition.

\subsubsection[D-structure of frame bundles]{$\boldsymbol{\mathcal D}$-structure of frame bundles}

The implicit derivation with respect to the standard coordinates in $\bigl(\mathbb C^n,0\bigr)$ induces $n$ canonical differential operators,
$\delta_i \colon \mathbb C(R_kM) \to \mathbb C(R_{k+1}M)$.

In order to see these differential operators as derivations of a~field, we take the inverse limit~${RM = \varprojlim R_kM}$ of the sequence
\[\cdots \to R_3M \to R_2M \to R_1M \to R_0M=M\]
as a~proalgebraic variety. Elements of $RM$ are formal frames that may or may not be convergent. $RM\to M$ is a~principal bundle modeled over the proalgebraic group
\smash{$\Gamma = \varprojlim \Gamma_k$} of formal automorphisms of $\bigl(\mathbb C^n,0\bigr)$. The field of rational functions in $RM$ is
\[\mathbb C(RM) = \bigcup_{k\geq 0} \mathbb C(R_kM)\]
and thus, the implicit derivative operators become derivations. By setting $\mathcal D = \{\delta_1,\dots,\delta_n\}$ the set of implicit derivative operators, we have that $(\mathbb C(RM), \mathcal D)$ is a~$\mathcal D$-field in the classical sense of Kolchin.

\subsubsection[Definition of D-groupoid]{Definition of $\boldsymbol{\mathcal D}$-groupoid}

Now let us define $\operatorname{Aut}(M)$ as the inverse limit of the sequence
\[
\cdots \to \operatorname{Aut}_3M \to \operatorname{Aut}_2M \to \operatorname{Aut}_1M \to \operatorname{Aut}_0M = M\times M.
\]
The elements of $\operatorname{Aut}(M)$ are formal biholomorphisms of $M$ on itself, which may or may not be convergent. By a~closed subset in $\operatorname{Aut}(M)$, we mean a~closed subset in the initial topology with respect to the projections $\operatorname{Aut}(M)\to\operatorname{Aut}_k(M)$. Such a~closed subset is a~sequence of Zariski closed subsets $Z = \{Z_k\}_{k\in \mathbb N}$ such that
\begin{enumerate}\itemsep=0pt
 \item[(a)] $Z_k\subseteq \operatorname{Aut}_kM$.
 \item[(b)] $Z_{k+1}$ dominates $Z_k$ by projection.
\end{enumerate}

\begin{Definition}\label{def_Dgroupoid} A~closed subset $\mathcal G = \{\mathcal G_k\}_{k\in \mathbb N} \subset \operatorname{Aut}(M)$ is a~$\mathcal D$-groupoid of transformations of~$M$ if and only if
\begin{enumerate}\itemsep=0pt
 \item[(a)] For all $k$, $\mathcal G_k$ is a~rational subgroupoid of $\operatorname{Aut}_kM$.
 \item[(b)] For all $f\in {\rm Inv}(\mathcal G_k)$ and $\delta \in \mathcal D$ we have $\delta f \in {\rm Inv}(\mathcal G_{k+1})$.
\end{enumerate}
\end{Definition}

This means that if $\mathcal G$ is a~$\mathcal D$-groupoid and we consider the field of invariants of any order
\[{\rm Inv}_{\mathcal D}(\mathcal G) = \bigcup_{k\geq 0} {\rm Inv}(\mathcal G_k),\]
condition (b) in Definition~\ref{def_Dgroupoid} is equivalent to the assertion that ${\rm Inv}_{\mathcal D}(\mathcal G)$ constitutes a~$\mathcal D$-field. We have a~correspondence again, between $\mathcal D$-groupoids and differential fields, given in~\cite[Proposition~2.16]{bcd2022malgrange}.

\begin{Proposition}
The assignment $\mathcal G \leadsto {\rm Inv}_{\mathcal D}(\mathcal G)$ is a~bijective correspondence between the set of $\mathcal D$-groupoids of transformations of $M$ and $\Gamma$-invariant $\mathcal D$-subfields of $\mathbb C(RM)$ containing $\mathbb C$.
\end{Proposition}

{From Kolchin theorem~\cite{kolchin}, any differential subfield of a~finitely generated differential field is finitely generated as a~differential field. In particular, there exists $k$ such that ${\rm Inv}_{\mathcal D}(\mathcal G)$ is generated by ${\rm Inv}(\mathcal G_k)$ as a~differential field}. This $k$ is the order of $\mathcal G$, and thus $\mathcal G$ is completely determined by $\mathcal G_{k}$.\footnote{The system of partial differential equations $\mathcal G_k\subset \operatorname{Aut}_k(M)$ determining $\mathcal G$ is known in the mathematical literature as a~Lie equation~\cite{kumpera1972lie,munoz2000integrability}.}

\subsubsection[D\_M-structure of Aut(M) and its Kolchin topology]{$\boldsymbol{\mathcal D_M}$-structure of $\boldsymbol{\operatorname{Aut}(M)}$ and its Kolchin topology}

In the original approach to $\mathcal D$-groupoid by B.~Malgrange~\cite{malgrange2001groupoide}, the definition is different. We give an overview of Malgrange definition and refer to~\cite{davy} for the proofs.

This alternative definition makes use of the $\mathcal D_M$-structure of $\operatorname{Aut}(M)$. Here, by $\mathcal D_M$ we understand the noncommutative ring of differential operators in $\mathbb C(M)$ spanned by the Lie algebra of rational vector fields $\mathfrak X(M)$. The main point here is that any vector field in $M$ can be extended to a~total derivative operator that differentiates functions of $\operatorname{Aut}_k(M)$ into functions of $\operatorname{Aut}_{k+1}(M)$. If $X$ is a~rational vector field in $M$, then
\[X^{\rm tot}\colon\ \mathbb C(\operatorname{Aut}_k(M))\to \mathbb C(\operatorname{Aut}_{k+1}(M)),
 \qquad \bigl(X^{\rm tot}f\bigr)\bigl(j_x^{k+1}\sigma\bigr) = \vec X_x\bigl(f\circ j^k\sigma\bigr).\]
In particular, derivations that are well defined in some $\mathcal O_M(U)$ will give us derivations of the ring of regular functions in $\operatorname{Aut}_k(U)$, this being an open subset of $\operatorname{Aut}_k(M)$.

For the sake of simplicity, let us assume that $M$ is an affine space, there are $n$ regular vector fields linearly independent at each point of $M$, that is, the module of derivations of $\mathcal O_M$ is free of rank $M$. The general case should be worked out locally.

An ideal $\mathcal J \subset \mathcal O_{\operatorname{Aut}}(M)$ is a~$\mathcal D_M$-ideal if for every function $f\in \mathcal J$ the total derivatives of $f$ are also in $\mathcal J$. A~subset $Z = \{Z_k\}_{k\in\mathbb N} \subset \operatorname{Aut}(M)$ is Kolchin closed if it is the set of zeros of a~$\mathcal D_M$-ideal. Kolchin closed subsets are also termed $\mathcal D_M$-varieties. An alternative, and historically original~\cite{malgrange2001groupoide}, definition for $\mathcal D$-groupoid goes as follows. \emph{A $\mathcal D$-groupoid in $M$ is a~$\mathcal D_M$-variety~${\mathcal G = \{\mathcal G_k\}_{k\in \mathbb N} \subseteq \operatorname{Aut}(M)}$ such that for all $k$, $\mathcal G_k$ is a~rational subgroupoid of $\operatorname{Aut}_k(M)$. }

\subsubsection[Lie algebra of a~D-groupoid]{Lie algebra of a~$\boldsymbol{\mathcal D}$-groupoid}

There is a~unique way to extend a~vector field $X$ on $M$, seen as a~derivations of $\mathbb C(M)$, to a~$\Gamma_k$ invariant derivation of each $\mathbb C(R_kM)$ that commutes with the canonical $\mathcal D$-structure $\mathcal D = (\delta_1,\dots,\delta_n)$. The extension $X^{(k)}$ is completely determined by the equations
$X^{(k+1)} \delta_j I = \delta_j X^{(k)}I$.
The flow of the vector field $X^{(k)}$ in $R_kM$ codifies the flow induced on the space of frames~$R_kM$ by the flow of $X$ on $M$. Moreover, the vector field $X^{(k)}$ is $\Gamma_k$-invariant. The dynamical interpretation of $X^{(\infty)}$ is just the flow induced by the flow of $X$ on the frame space, in the sense
\[X^{(k)}\bigl(j^k_0r_x\bigr) = \left.\frac{{\rm d}}{{\rm d} t}\right\vert_{t=0} j^k_0(\exp(tX)\circ r_x).\]

In the product $R^kM\times R^kM$ we may consider the vector field $\bigl(0,X^{(k)}\bigr)$ that represents $X^{(k)}$ acting on the second component. This vector field is $\Gamma_k$-invariant, and it has a~projection by the canonical quotient map \smash{$R^kM\times R^kM\to \operatorname{Aut}_k(M)$}, \smash{$\bigl(j_0^kr_x,j_0^kr_y\bigr)\mapsto j_x^k\bigl(r_y\circ r_x^{-1}\bigr)$}. We define $\widetilde X^{(k)}$ as its projection on $\operatorname{Aut}_k(M)$. The dynamical interpretation of $\widetilde X^{(k)}$ is the post-composition with the flow of $X$
\[\widetilde X^{(k)}\bigl(j_x^k \sigma\bigr) = \left.\frac{{\rm d}}{{\rm d} t}\right\vert_{t=0} j_x^k((\exp tX)\circ \sigma).\]

Let us consider the linear bundle $J^k(TM/M)\to M$ of $k$-jets of vector fields. The value of~$\widetilde X^{(k)}$ at $e_x = j^k_x{\rm id}$ depends only on the $k$-jet of $X$ at $x$. Therefore, we have a~well-defined regular map
\[\iota_k \colon\ J^{k}(TM/M) \to T(\operatorname{Aut}_k(M)), \qquad j^k_xX \mapsto \widetilde X^{(k)}(e_x).\]
This map identifies the space $\bigl(J^k(TM/M)\bigr)_x$ with $T_{j_x{\rm Id}}(\operatorname{Aut}(M)_{x\bullet})$.\footnote{This is the usual definition of the Lie algebroid of the Lie groupoid $\operatorname{Aut}_k(M)$, and thus $J^k(TM/M)\to M$ can be seen as a~realization of the Lie algebroid of $\operatorname{Aut}_k(M)$. However, the Lie algebroid bracket is somehow different from the Lie algebra structure that we find in the sections of ${\rm Lie}(\mathcal G)$.}

However, this prolongation process can be carried out to infinite order, so that from a~vector field $X$ in $M$ we obtain a~$\Gamma$-invariant derivation $X^{(\infty)}$ of the field $\mathbb C(RM)$ and a~derivation $\widetilde X^{(\infty)}$ of the field $\mathbb C(\operatorname{Aut}(M))$. Then, we have the limit
\[\varprojlim {\rm Lie}(\mathcal G_k) = {\rm Lie}(\mathcal G) \subset J(TM/M), \]
which is, on the generic point of $M$, an involutive linear system of partial differential equations in the linear bundle $TM/M$ and a~bundle of Lie algebras.

\begin{Definition}For a~$\mathcal D$-groupoid $\mathcal G$, we denote by
 ${\rm Lie}(\mathcal G_k) = \iota_k^{-1}(T\mathcal G_k) \subset J^{k}(TM/M)$, and by~\smash{${{\rm Lie}(\mathcal G) = \varprojlim {\rm Lie}(\mathcal G_k) \subset J(TM/M)}$}.
\end{Definition}

\smash{${\rm Lie}(\mathcal G_k) = \iota_k^{-1}(T\mathcal G_k)$} is an algebraic family vector spaces ${\rm Lie}(\mathcal G_k)\to M$ which at the generic point is a~vector bundle over $M$ whose rank is the dimension of $\mathcal G_k$ minus $n$. The fiber of ${\rm Lie}(\mathcal G_k)$ at $x$ is the vector space of $k$-jets of vector fields $X$ having the property
\[\left.\frac{{\rm d}}{{\rm d} t}\right\vert_{t=0} j_x^k(\exp tX) \in T\mathcal (G_k)_{x\bullet}.\]
That is, the flow of $X$ near $x$ gives a~local parametric group of transformations which is \emph{tangent to $\mathcal G_k$.}

Note that the $k$-jet at $x$ of a~Lie bracket $[X,Y]$ depends on the $(k+1)$-jets at $x$ of the vector fields $X$ and $Y$. Therefore, we do not obtain a~bundle of Lie algebras, our Lie bracket takes values on jets of smaller order
\[{\rm Lie}(\mathcal G_{k+1}) \times_M {\rm Lie}(\mathcal G_{k+1}) \to {\rm Lie}(\mathcal G_{k}),
 \qquad \bigl(j^{k+1}_x X, j^{k+1}_x Y\bigr) \mapsto j_x^k[X,Y].\]
However, this bracket defines a~Lie bracket on fibers of ${\rm Lie}(\mathcal G)$ called the fiberwise bracket.
The term $\mathcal D$-Lie algebra will be used for ${\rm Lie}(\mathcal G)$ endowed with the fiberwise bracket. As a~Lie algebroid, ${\rm Lie}(\mathcal G_k)$ is endowed with a~bracket on its sheaf of sections. This bracket is known as the Spencer bracket. See~\cite{malgrange2001groupoide} for a~comparison of the two brackets.

The equations defining ${\rm Lie}(\mathcal G)$ can be found in terms of the differential invariants of $\mathcal G$
\[{\rm Lie}(\mathcal G) = \big\{j_x X \in J(TM/M) \mid X^{(\infty)}{\rm Inv}_{\mathcal D}(\mathcal G)=0\big\}.
\]
Note that if the differential invariants of $\mathcal G$ are preserved by $X^{(\infty)}$ then the flow of $\widetilde X^{(\infty)}$ in~$\operatorname{Aut}(M)$ is tangent to $\mathcal G$.

\subsection[Finite-dimensional D-groupoids]{Finite-dimensional $\boldsymbol{\mathcal D}$-groupoids}

In this paper, we are interested in $\mathcal D$-groupoids of finite dimension.

\begin{Definition} A~$\mathcal{D}$-groupoid $\mathcal{G} = \{\mathcal{G}_m\}_{m\in \mathbb N}\subset \operatorname{Aut}(M)$ is said to be of finite dimension if its sequence of projections stabilizes at some integer $\ell \in \mathbb N$, meaning that for all $m \geq \ell$, the projections $\mathcal{G}_{m+1}\to \mathcal{G}_m$ are isomorphisms, in this case, we have
 \[\mathcal G \simeq \cdots \simeq \mathcal G_{\ell+1} \simeq \mathcal G_\ell \to \mathcal G_{\ell-1} \to \cdots \to M\times M.\]
\end{Definition}

\begin{Remark}
 It is important to distinguish the stabilization step $\ell$ from the intrinsic order of the $\mathcal{D}$-groupoid (the order of the system of differential equations determining it). Throughout this paper, for a~finite-dimensional $\mathcal{D}$-groupoid $\mathcal{G}$ of order $q$, we will always fix an integer~${k \geq {\rm max}(\ell, q)}$. This choice ensures that $\mathcal{G}$ is both completely determined by and faithfully represented by its $k$-th component, $\mathcal{G} \simeq \mathcal{G}_k$.
\end{Remark}
\begin{Definition}
 Let $\mathcal{G}$ be a~finite-dimensional $\mathcal{D}$-groupoid. A~point $x\in M$ is regular for $\mathcal G$ if there is a~Zariski open neighborhood $M'\subset M$ of $x$ such that $\mathcal G|_{M'}\subseteq \operatorname{Aut}(M')$ is a~smooth Lie groupoid.
\end{Definition}

Given a~finite-dimensional groupoid $\mathcal G\subset \operatorname{Aut}(M)$, the dimension of $\mathcal G$ over $M$ is the dimension of the fibers of the source $s$, or the target $t$ projection. We have $\dim_M(\mathcal G) = \dim(\mathcal G)-\dim(M)$. We will say that $\mathcal G$ is $0$-dimensional or finite if $\dim_M(\mathcal G) = 0$.

\subsubsection{Adapted frames}

Let $\mathcal G\subset \operatorname{Aut}(M)$ be a~finite-dimensional transitive\footnote{In the sense that $\mathcal G_{0} = M\times M$.} $\mathcal D$-groupoid. We also fix and order $k$ large enough so that $\mathcal G\simeq \mathcal G_k$, and such that $\mathcal G$ is determined by its differential invariants of order $\leq k$. Without loss of generality, after restriction to a~suitable open subset, we also assume that $\mathcal G_k$ is a~Lie subgroupoid of $\operatorname{Aut}_k(M)$.

From now on, we identify $\mathcal G$ with its truncation to order $k$. However, it is important to remember that any element $\sigma_{xy}\in \mathcal G$ is an actual germ of local biholomorphism ${\sigma_{xy}\colon (M,x)\to (M,y)}$, completely determined by its $k$-jet $j^k_x\sigma_{xy}\in \mathcal G_k$.

Let us fix a~point $x$, and lets recall that $\mathcal G_{xx}$ is the stabilizer group of $x$ in $\mathcal G$.\footnote{Since $\mathcal G$ is transitive, all the points have isomorphic stabilizer groups.}
In the transitive case, the moving frame method gives us an effective way of computing all differential invariants. Note that the group $\mathcal G_{xx}$ acts on the space of frames on $x$, $(R_kM)_x$ by composition. Let us fix a~frame $j_0^kr_x \in (R_kM)_x$. It induces isomorphisms
\begin{alignat*}{3}
& i_1\colon\ \Gamma_k \to (R_kM)_x, \qquad&& j_0^k\gamma \mapsto j_0^k(r_x \circ \gamma),&
\\
& i_2 \colon\ \Gamma_k \to \operatorname{Aut}_k(M)_{xx}, \qquad&& j_0^k\gamma \mapsto j_x^k\bigl(r_x \circ \gamma\circ r_x^{-1}\bigr)&
\end{alignat*}
that conjugate the action of $\mathcal G_{xx}$ in $(R_kM)_x$ with the action of the algebraic subgroup $\Gamma_k' = j_0^k\bigl(r_{x}^{-1}\mathcal G_{xx}r_x\bigr) \subset \Gamma_k$ on $\Gamma_k$. Therefore, we have the existence of the quotient $V = (R_kM)_x/\mathcal G_{xx} \simeq \Gamma_k/\Gamma_k'$.

\begin{Proposition}\label{prop_dinvariant}
 The map $I\colon R_kM \to V$, $ j_0^kr_y \mapsto [\sigma_{xy}\circ r_y]$,
where $\sigma_{xy}$ is any element of $\mathcal G_{xy}$, is a~well-defined regular map, invariant by the action of $\mathcal G$, and it identifies ${\rm Inv}(\mathcal G)$ with $\mathbb C(V)$. For any $v\in V$, the preimage $I^{-1}(\{v\})$ with its projection to $M$ is a~reduction of the bundle~${R_kM\to M}$ to a~subgroup of $\Gamma_k$ isomorphic to $\mathcal G_{xx}$. All these principal bundles are isomorphic.
\end{Proposition}

\begin{Definition}
Let $\mathcal G$ be a~finite-dimensional $\mathcal D$-groupoid such that $\mathcal G\hookrightarrow \operatorname{Aut}_k(M)$. Let us assume that $\mathcal G$ is also a~smooth Lie groupoid. Let $I\colon R_kM \to V$ be a~complete differential invariant map as in Proposition~\ref{prop_dinvariant}. Each bundle of the form $I^{-1}(\{v\})\to M$ from $v\in V$ is called a~bundle of adapted frames for $\mathcal G$.
\end{Definition}

\begin{Remark}
From~\cite{Casale2004SurLG}, it follows that this phenomenon actually represents the general situation for transitive Lie groupoids: every such groupoid arises as the groupoid of symmetries of a~meromorphic geometric structure.
\end{Remark}

\subsubsection[D-Lie algebra and Lie algebroid]{$\boldsymbol{\mathcal D}$-Lie algebra and Lie algebroid}

Here we assume that $\mathcal G$ is of finite dimension $\dim_M\mathcal G = r$, and $\mathcal G \simeq \mathcal G_k$ is a~smooth Lie subgroupoid of $\operatorname{Aut}_k(M)$. Therefore, we have ${\rm Lie}(\mathcal G) \simeq {\rm Lie}(\mathcal G_k) \subset J^k(TM/M)$, and then ${\rm Lie}(\mathcal G)\to M$ is a~linear bundle of rank $r$ of Lie algebras. The contact system on the jet bundle $T\bigl(J^kM/M\bigr)$ induces by restriction to ${\rm Lie}(\mathcal G)\to M$ a~flat linear connection $\nabla^{{\rm Lie}(\mathcal G)}$, and it is stable under the Lie bracket in the sense that the Lie bracket of any two horizontal sections is also a~horizontal section.

Horizontal sections of $\nabla^{{\rm Lie}(\mathcal G)}$ are of the form $jX\colon x\mapsto j_xX$, where $X$ is a~vector fields in $M$ whose flow is in $\mathcal G$. We say that $X$ is an infinitesimal generator of $\mathcal G$. The elements of the Lie algebra ${\rm Lie}(\mathcal G)_x$ are the germs at $x$ of the infinitesimal generators of $\mathcal G$.

\begin{Remark}
A~reader familiar with Lie algebroid perspective will notice that $\bigl({\rm Lie}(\mathcal G), \nabla^{{\rm Lie}(\mathcal G)}\bigr)$ is a~locally symmetric Cartan algebroid in the sense of~\cite[Section~2.5]{blaom2006geometric} or a~twisted Lie algebra action as in~\cite[Section~2.5]{Blaom2016b}.
\end{Remark}

\subsubsection{Example: Projective structures on the complex line}

Let $M= \mathbb C$ be the complex line, with coordinate $\lambda$. By a~projective structure {on} some open subset $M'\subset M$ we mean a~maximal atlas in $M$ with values {in} the projective line $\overline{\mathbb C}$ such that the transition functions are M\"obius transformations on their domain in $\overline{\mathbb C}$. The chart functions~${\tau\colon \mathcal U \to \overline{\mathbb C}}$ should satisfy the differential equation
\begin{equation}\label{eq_s1}
S_\lambda(\tau) = R(\lambda), \qquad S_{\lambda}(\tau) = \frac{\tau_{\lambda\lambda\lambda}}{\tau_\lambda} - \frac{3}{2}\left(\frac{\tau_{\lambda\lambda}}{\tau_{\lambda}}\right)^2,
\end{equation}
where $R(\lambda)$ is defined on $M'$. We say that the projective structure is meromorphic, or rational, if $R(\lambda) \in \mathbb C(\lambda)$ is a~rational function. Equation~\eqref{eq_s1} is equivalent to
\begin{equation}\label{eq_s2}
S_\tau(\lambda) + R(\lambda)\lambda_{\tau}^{2} = 0.
\end{equation}
Now, let us consider $\operatorname{Aut}_3(\mathbb C)$, with coordinates $\lambda$, $\varphi$, $\varphi_{\lambda}$, $\varphi_{\lambda\lambda}$, $\varphi_{\lambda\lambda\lambda}$. Let us derive the equations for the symmetries $\varphi$ of~\eqref{eq_s2}, in the sense that $\lambda(\tau)$ is a~solution of~\eqref{eq_s2}, if and only if $\varphi(\lambda(\tau))$ is also a~solution. By the chain rule for the Schwarzian derivative, we obtain
\begin{equation}\label{eq_kummer}
S_\lambda(\varphi) + \varphi_\lambda^2R(\varphi)=R(\lambda).
\end{equation}

This is the equation of a~groupoid $\mathcal G_3 \subset \operatorname{Aut}_3(\mathbb C)$. Let us compute the third-order differential invariant giving rise to the adapted frame bundles.

Let us consider a~frame $(\lambda,\lambda_\varepsilon,\lambda_{\varepsilon\varepsilon},\lambda_{\varepsilon\varepsilon\varepsilon})$. If we consider a~$3$-jet of biholomorphism $(\lambda,\varphi,\varphi_\lambda,\allowbreak \varphi_{\lambda\lambda},\varphi_{\lambda\lambda\lambda})$, the composition gives us a~new frame $(\varphi,\varphi_\varepsilon,\varphi_{\varepsilon\varepsilon},\varphi_{\varepsilon\varepsilon\varepsilon})$ where, as before for the chain rule
$
S_{\varepsilon}(\varphi) = S_{\lambda}(\varphi) \lambda_{\varepsilon}^2 + S_\varepsilon(\lambda)$,
and replacing $S_{\lambda}(\varphi)$ from the groupoid equation~\eqref{eq_kummer}, we obtain
\[
S_{\varepsilon}(\varphi) + R(\varphi)\varphi_\varepsilon^2 = S_{\varepsilon}(\lambda) + R(\lambda)\lambda_\varepsilon^2,
\]
and therefore
\begin{equation}\label{eq_invariant}
I(\lambda,\lambda_{\varepsilon},\lambda_{\varepsilon\varepsilon},\lambda_{\varepsilon\varepsilon\varepsilon}) = R(\lambda)\lambda_\varepsilon^2 + \frac{\lambda_{\varepsilon\varepsilon\varepsilon}}{\lambda_\varepsilon} - \frac{3}{2}
\left(\frac{\lambda_{\varepsilon\varepsilon}}{\lambda_\varepsilon}\right)^2
\end{equation}
is a~differential invariant for the groupoid $\mathcal G_3$. Let us compute the equation for the Lie algebra of $\mathcal G_3$. Let $X = a(\lambda)\frac{\partial}{\partial \lambda}$ be a~vector field in $\mathbb C$. Its third-order prolongation $X^{(3)}$ is then
\begin{align*}
X^{(3)} ={} & a\frac{\partial}{\partial \lambda} + a_\lambda\lambda_\varepsilon\frac{\partial}{\partial \lambda_\varepsilon}+
\bigl(a_{\lambda\lambda} \lambda_\varepsilon^2 + a_\lambda \lambda_{\varepsilon\varepsilon}\bigr)\frac{\partial}{\partial \lambda_{\varepsilon\varepsilon}}
+ \bigl(a_{\lambda\lambda\lambda}\lambda_\varepsilon^3 + 3a_{\lambda\lambda}\lambda_\varepsilon\lambda_{\varepsilon\varepsilon}
+ a_\lambda\lambda_{\varepsilon\varepsilon\varepsilon} \bigr) \frac{\partial}{\partial \lambda_{\varepsilon\varepsilon\varepsilon}},
\end{align*}
and a~direct computation gives us
\[
X^{(3)}I = (a_{\lambda\lambda\lambda} + 2R(\lambda) a_{\lambda} + R'(\lambda)a)\lambda_\varepsilon^2.
\]
So that, we find the linear differential equation defining the Lie algebra of $\mathcal G_3$,
\begin{equation}\label{eq_sym2}
 a_{\lambda\lambda\lambda} + 2R(\lambda)a_{\lambda} + R'(\lambda)a = 0.
\end{equation}

\subsection{Picard--Vessiot theory}

There are several presentations of Picard--Vessiot theory, also known as linear differential Galois theory. In~\cite[Appendix~A]{bc2017}, we provided a~geometrical presentation that is equivalent to Kolchin's theory of strongly normal extensions of finite transcendence degree over~$\mathbb C$. We recall here the main points of the theory in a~simplified way, that applies to the examples and applications developed in this paper. We direct the reader to the aforementioned reference for a~full exposition.

\subsubsection{Definition of the Galois group}

Let us consider $G$ an algebraic group, $\pi\colon P\to M$ a~connected principal bundle modeled over $G$, and a~flat rational connection with connection form $\theta\colon P \dasharrow {\rm Lie}(G)$.
Let us call $M'$ a~Zariski open subset of $M$ such that $\theta$ is regular on $P|_{M'} = P'$. Then $\ker(\theta) = \mathcal F \subset TP'$ is a~foliation on $P'$. For a~given point $p\in P'$ let $Z_p = \overline{\mathcal L_p}$ be the Zariski closure of the leave of $\mathcal F$ passing through $p$. Then, we can define a~subgroup of~$G$,
\[
{\rm Gal}(\theta, p) = \{g\in G \mid Z_pg = Z_p \} \subseteq G,
\]
which is called the Galois group of the connection given by $\theta$ at the point $p$. It has the following properties:
\begin{enumerate}\itemsep=0pt
\item[(1)] ${\rm Gal}(\theta, p)$ is an algebraic subgroup of $G$.
\item[(2)] $Z_p \to M'$ is an irreducible principal bundle modeled over ${\rm Gal}(\theta, p)$.
\item[(3)] All Galois groups are isomorphic and belong to the same conjugation class as subgroups of $G$.
\end{enumerate}

Thus, we may write ${\rm Gal}(\theta)$ to refer to the Galois group, as an abstract algebraic group independent of the point $p$ in $P'$. We also write $\mathfrak{gal}(\theta)$ to refer to the Lie algebra of the Galois group. Note that $\mathfrak{gal}(\theta) = 0$ if and only if the Galois group is finite.

\subsubsection{Galois group for linear connections}

The above construction applies to linear connections in the following way. Let us consider a~linear bundle $E\to M$ of rank $k$, and $\nabla$ a~flat linear rational connection. As previously, we consider $M'$ the domain of the connection and $E'$ the restriction of $E$ to $M'$, so that, $\nabla$ is a~regular connection over $M'$.
Let $E_x$ be the fiber at $x \in M'$. The bundle $P \subset E_x^\ast \otimes E$ whose fiber at $x' \in M'$ is the set of linear isomorphisms from $E_x$ to $E_{x'}$ is a~$\operatorname{GL}(E_x)$-principal bundle over $M'$. The connection $\nabla$ induces a~principal connection with connection form
\[
\theta_\nabla(X_p) = p^{-1} \circ \nabla_{\pi_*(X_p)} p \in \mathfrak{gl}(V),
\]
where $p \colon E_x \to E_{x'}$ is a~linear isomorphism and $\pi_*(X_p)$ is the projection of the vector $X_p\in T_p(P)$ on $M'$. Now $P$ has a~special point, ${\rm id} \colon E_x \to E_x$ and ${\rm Gal}(\theta_\nabla, {\rm id})$, denoted by ${\rm Gal}(\nabla, x)$, is the Galois group of $\nabla$ at $x$. Note that ${\rm Gal}(\nabla, x)\subset \operatorname{GL}(E_x)$, so that the space $E_x$ is naturally endowed with a~representation of the Galois group ${\rm Gal}(\nabla, x)$.

\begin{Remark}\label{rmk_products}
An important property of this construction is that it is compatible with direct product. If $E_i\to M_i$
with $i=1,\dots, k$ is a~finite family of vector bundles endowed with flat rational linear connections $\nabla_i$, then we may consider the direct product connection $\nabla = \prod \nabla_i$ in the direct product bundle $E = \prod E_i \to M = \prod M_i$, and in such case ${\rm Gal}(\nabla, (x_1,\dots,x_k)) = \prod {\rm Gal}(\nabla_i, x_i)$, acting on $E_{(x_1,\dots,x_k)}= \prod (E_i)_{x_i}$
\end{Remark}

\subsection{Parallelisms}

This section is a~recap of the main concepts and results presented in~\cite{bc2017}. Let $\mathfrak g$ be a~complex Lie algebra, and $M$ a~smooth complex irreducible algebraic variety over $\mathbb C$ with dim $M = n$.

\begin{Definition} A~\emph{$\mathfrak g$-parallelism} in $M$ is a~Lie algebra morphism $\varphi\colon \mathfrak g \to \mathfrak X(M)$ such that for all $x$ in $M$ the linear map ${\rm ev}_x\circ \varphi\colon \mathfrak g \to T_xM$, $A\mapsto \varphi(A)_x$ is an isomorphism.
The image of~${\varphi(\mathfrak g) \subset \mathfrak X(M)}$ is the Lie algebra of parallel vector fields of $\varphi$. A~parallelism in $M$ is a~$\mathfrak g$-parallelism for some Lie algebra $\mathfrak g$.
\end{Definition}

By duality, any $\mathfrak g$-parallelism can be codified into a~\emph{$\mathfrak g$-parallelism form} $\omega\colon TM\to \mathfrak g$ by defining $\omega(X_x) = ({\rm ev}_x\circ \varphi)^{-1}(X_x)$. It is possible to give the following definition. It follows that any $\mathfrak g$-valued $1$-form $\omega$ in $M$ is a~parallelism form if and only if it satisfies
\begin{itemize}\itemsep=0pt
 \item[(1)] For all $x\in M$, $\omega_x\colon T_xM \to \mathfrak g$ is a~linear isomorphism.
 \item[(2)] If $A$, $B$ are elements of $\mathfrak g$ and $X$, $Y$ are the vector fields in $M$ such that $\omega(X) =A$ and $\omega(Y) = B$, then $\omega([X,Y]) = [A,B]$, or equivalently:
 \item[(2')] It satisfies the Maurer--Cartan equation ${\rm d}\omega + \frac{1}{2}[\omega,\omega] = 0$.
\end{itemize}

With this dual definition, a~parallel vector field with respect the parallelism form $\omega$ is a~vector field $X$ with constant $\omega(X)\in \mathfrak g$. A~$\mathfrak g$-parallelism in $M$ is considered to be regular, rational, analytic, formal, etc., depending on the degree of regularity of its corresponding parallelism form. A~\emph{$\mathfrak g$-parallelized variety} is then a~pair $(M,\omega)$, where $\omega$ is a~$\mathfrak g$-parallelism form, and a~conjugation of parallelized varieties $f\colon (M,\omega)\to (N,\varpi)$ is a~map $f\colon N\to M$ such that $f^*(\varpi) = \omega$. As before we may consider regular, rational, analytic or formal morphisms.

The main example of parallelized variety is $(G, \omega_G)$, where $G$ is an algebraic group and $\omega_G\colon TG\to {\rm Lie}(G) = T_eG$ is its structure for $\omega_G(X_g) = {\rm d}L_g^{-1}(X_g)$. Parallel vector fields of $\omega_G$ are left invariant vector fields. Note that there is a~reciprocal parallelism defined by ${\omega_G^{\rm rec}(X_g) = - {\rm d}R_{g}^{-1}(X_g)}$, whose parallel vector fields are the right invariant vector fields in $G$. Note that the inversion $(G,\omega_G) \to \bigl(G,\omega_G^{\rm rec}\bigr)$, $g\mapsto g^{-1}$ is a~regular conjugation between the parallelisms $\omega_G$ and $\omega_G^{\rm rec}$.

Given varieties $(M,\omega)$ and $(N,\varpi)$ with rational parallelisms, we are interested in studying birational conjugations, but also multivalued algebraic conjugations $f\colon M\dasharrow N$. A~way of understanding these multivalued algebraic maps is through the notion of \emph{isogeny}. We say that~$(M,\omega)$ and $(N,\varpi)$ are isogenous if there is an irreducible parallelized variety $(U,\theta)$ and dominant maps $\pi_M \colon U\to M$, $\pi_N\colon U \to N$ such that $\pi_M^*(\omega) = \pi_N^*(\varpi) = \theta$. We may assume that $U$ is a~covering of some open subvarieties of $M$ and $N$. This notion extends the classical notion of isogeny from algebraic groups.

\begin{Example}
 Let $(G,\omega_G)$ be a~connected algebraic group with its structure form. Let $H\subset G$ be a~finite group, so that the form $\omega_G$ is projectable to the quotient
 \[\pi\colon M = G\to H\backslash G = \{Hg \mid g\in G\}\]
 then $(M, \pi_*(\omega_G))$ is isogenous to $(G,\omega_G)$. Moreover, let $f\colon N\dasharrow M$ be a~dominant map, with $N$ irreducible variety of the same dimension as $M$. Then $\bigl(N,f^*(\pi_*(\omega_G))\bigr)$
 is isogenous to~$(G,\omega_G)$.
\end{Example}

\subsection{Lie connections}

In this section, we deal with affine connections in $M$, that is, linear connections on $TM\to M$. First, let us note that given an affine connection $\nabla$ we can always define a~reciprocal connection by the formula
$\nabla^{\rm rec}_X Y = \nabla_Y X + [X,Y]$.
From this definition, the difference $\nabla - \nabla^{\rm rec}= {\rm Tor}_\nabla$ is the torsion tensor. Let us now consider the following definition, taken from~\cite{bc2017}.

\begin{Definition} A~Lie connection in $M$ is a~flat connection $\nabla$ in $TM$ such that the Lie bracket of two horizontal vector fields is also horizontal.
\end{Definition}

In other words, a~Lie connection is a~linear differential equation that locally defines a~parallelism. We may say that a~Lie connection is modeled over the Lie algebra $\mathfrak g$ if the space of horizontal vector fields defined around any regular point is a~Lie algebra isomorphic to $\mathfrak g$. Lie connections can be characterized in the following way, as is done in~\cite[Proposition 3.10]{bc2017}.

\begin{Proposition}
Let $\nabla$ be a~Lie connection in $TM$, the following statements are equivalent:
\begin{enumerate}\itemsep=0pt
 \item[$(1)$] $\nabla$ is a~Lie connection.
 \item[$(2)$] $\nabla^{\rm rec}$ is a~Lie connection.
 \item[$(3)$] $\nabla$ is flat and has constant torsion in the sense $\nabla {\rm Tor}_\nabla = 0$.
 \item[$(4)$] $\nabla$ and $\nabla^{\rm rec}$ are flat.
\end{enumerate}
\end{Proposition}

There is an interesting relationship between a~pair of reciprocal Lie connections $\nabla$ and $\nabla^{\rm rec}$. They are modeled over the same Lie algebra. Horizontal vector fields for $\nabla$ commute with the horizontal vector fields of $\nabla^{\rm rec}$ and vice-versa. We are in an infinitesimal situation that is analogous to that of the Lie algebras of left and right invariant vector fields in an algebraic group.

Given a~$\mathfrak g$-parallelism form $\omega$, we can always find the unique Lie connection $\nabla$ whose horizontal vector fields are those that are parallel with respect to $\omega$. Given a~basis $\{A_1,\dots,A_n\}$, we can fix a~frame $(X_1,\dots,X_r)$ in $M$ defined by $\omega(X_i)=A_i$, and then we define $\nabla$ such that $\nabla_{X_i}X_j = 0$ for all $1 \leq i$, $j \leq n$.
Then $\nabla^{\rm rec}$ is also a~Lie connection modeled over the same Lie algebra $\mathfrak g$ and its horizontal vector fields are the infinitesimal symmetries of $\omega$, in the sense that $\nabla^{\rm rec} Y = 0$ if and only if $\mathcal{L}_Y\omega = 0$.

\begin{Theorem}[{\cite[Corollary 4.3\,(a) and (d)]{bc2017}}]\label{thm:0}
Let $\mathfrak g$ be a~centerless Lie algebra. An algebraic variety $(M,\omega)$ with a~rational parallelism of type $\mathfrak g$ is isogenous to an algebraic group if and only if $\mathfrak{gal}\bigl(\nabla^{\rm rec}\bigr) = \{0\}$. Moreover, if ${\rm Gal}\bigl(\nabla^{\rm rec}\bigr) = \{ \rm id \}$ then the isogeny is given by a~dominant rational map $\pi\colon M \dasharrow G$ such that $\pi^*(\omega_G) = \omega$.
\end{Theorem}

\section[Isogeny and integrability of D-groupoid equations]{Isogeny and integrability of $\boldsymbol{\mathcal D}$-groupoid equations}

Here we want to know how far a~finite-dimensional $\mathcal D$-groupoid is from being \emph{algebraically integrable}, in the sense of getting algebraic formulae for biholomorphisms of $M$ that are sections of $\mathcal G$. For this, we extend the notion of isogeny of parallelisms to a~general notion of isogeny of $\mathcal D$-groupoids. If a~$\mathcal D$-groupoid is isogenous to an algebraic group action, then for each element~${j_x\sigma\in \mathcal G}$ there is a~unique algebraic map $\sigma$ whose jet is $j_x\sigma$.

On the other hand, we have $\mathcal D$-groupoids that are far to be algebraically integrable. A~common characteristic of any algebraic group $G$ is that the diagonal is always a~subgroup of $G\times G$. However, for $\mathcal D$-groupoids, the diagonal is not always a~$\mathcal D$-groupoid. We will see that far from integrable $\mathcal D$-groupoids have no nontrivial sub-$\mathcal D$-groupoids in their cartesian powers.

\begin{Definition}
Let \smash{$\pi\colon \widetilde{M} \to M$} be a~dominant map between irreducible
varieties of the same dimension $n$, and $\mathcal G\subset \operatorname{Aut}(M)$ a~$\mathcal D$-groupoid in $M$.
The pullback \smash{$\pi^*\mathcal G = \widetilde{\mathcal G} \subset\operatorname{Aut}(\widetilde{M})$}
is the $\mathcal D$-groupoid in \smash{$\widetilde M$} whose field of invariants is \smash{${\rm Inv}_{\mathcal D}\big(\widetilde{\mathcal G}\big) = \pi^*({\rm Inv}_{\mathcal D}(\mathcal G))$}.
\end{Definition}

After restriction to suitable open subset of \smash{$\widetilde M$} and $M$, we may assume that $\pi$ is a~covering. Then, it also induces coverings $\pi_*\colon R\widetilde M \to RM$ here $\pi_*(r_{\widetilde x})= \pi \circ r_{\widetilde x}$ and $\pi_\star\colon \operatorname{Aut}(\widetilde M) \to \operatorname{Aut}(M)$ defined by
$\pi_{\star}\bigl(r_{\widetilde y} \circ r_{\widetilde x}^{-1}\bigr) =
\pi_{\star}\bigl(r_{\widetilde y} \circ r_{\widetilde x}^{-1}\bigr) = (\pi\circ r_{\widetilde y}) \circ \bigl(\pi \circ r_{\widetilde x}^{-1}\bigr)$. Here we use a~decomposition of a~jet of biholomorphism as the product of two frames, equivalently we can set $\pi_\star(j_{\widetilde x}\sigma) = j_x (\pi\circ \sigma \circ \widetilde\pi)$
where $\widetilde\pi$ is the local inverse of $\pi$ that makes $\widetilde\pi(\pi(\widetilde x)) = \widetilde x$. In such case, the pullback of $\mathcal G$ is just the preimage $\widetilde{\mathcal G} = \pi_\star^{-1}(\mathcal G)$ and in the transitive case, a~bundle of adapted frames for $\mathcal G$ is given by $\pi_{*}^{-1}(R_{\mathcal G} M)$, where $R_{\mathcal G} M$ is a~bundle of adapted frames for $\mathcal G$.

\begin{Definition}\label{def_isogeny}
Let $M$ and $N$ be complex irreducible smooth varieties of dimension $n$. We say that two $\mathcal D$-groupoids $\mathcal G \subset \operatorname{Aut}(M)$ and $\mathcal H \subset\operatorname{Aut}(N)$ are isogenous if there is a~complex irreducible smooth variety \smash{$\widetilde M$} and finite dominant maps \smash{$\pi_1\colon \widetilde M\to M$}, \smash{$\pi_2\colon \widetilde M \to N$} such that~${\pi_1^\ast(\mathcal G) = \pi_2^*(\mathcal H)}$.
\end{Definition}

\begin{Example}
Definition~\ref{def_isogeny} extends the definition of isogeny of parallelisms introduced in~\cite[Definition 2.7]{bc2017}. If $(M,\omega)$ is a~variety with parallelism, then we can define its $\mathcal D$-groupoid of symmetries
\[{\rm Sym}(M,\omega) = {\rm cl}\{j_x\sigma \in \operatorname{Aut}(M)\mid \sigma^*(\omega) = \omega \}.\]
If $(N,\theta)$ is also a~parallelized variety, then it happens that $(M,\omega)$ and $(N,\theta)$ are isogenous if and only if ${\rm Sym}(M,\omega)$ and ${\rm Sym}(N,\theta)$ are isogenous $\mathcal D$-groupoids.

However, it is not true that a~groupoid $\mathcal H$ on $N$ isogenous to ${\rm Sym}(\omega)$ is the $\mathcal D$-groupoid of symmetries of a~parallelism as the latter as a~trivial isotropy groups and isotropy groups are not preserved by isogeny.

If a~transitive $\mathcal D$-groupoid $\mathcal H$ on $N$ has dimension $n+0$ and $x \in N$, then $t \colon \mathcal H_x \to N$ is a~finite surjective map. Then the form $\omega \colon T(\mathcal H_x) \to {\rm Lie}(\mathcal H)_x$ defined by $\omega(\sigma, v) = {\rm d}\sigma^{-1}({\rm d}t(v))$ is a~parallelism form with values in the Lie algebra of $\mathcal H$ at $x$.
In this situation, $t$ is a~isogeny between $\mathcal H$ and ${\rm Sym}(\mathcal H_x,\omega)$
\end{Example}

\subsection{The integrable example: action groupoids}\label{ss_actiongroupoids}

Let us consider an algebraic group $G$ with Lie algebra $\mathfrak g$ of left invariant vector fields acting faithfully and transitively on the homogeneous space $M = G/H$. We can consider the action groupoid,
$G\ltimes M \to M\times M$, $ (g, x) \mapsto (x, g\cdot x)$.
There is a~natural embedding that represents~${G\ltimes M}$ in $\operatorname{Aut}(M)$,
$j\colon G\ltimes M \hookrightarrow \operatorname{Aut}(M)$, $ (g, x) \mapsto j_x L_g$,
where $L_g\colon M\to M$ is the left translation $L_g(x) = g\cdot x$. Then $\mathcal G$ is a~$\mathcal D$-groupoid of finite dimension $n + r$ where $r = {\dim}(G)$.

Let us fix $x\in M$ and let us consider $\mathcal G_{x\bullet} = s^{-1}(x)$ the set of elements of $\mathcal G$ with source $x$. Then we have
$
\mathcal G_{x\bullet} = j(G\times \{x\} ) \simeq G \times \{x\}$.

On $G$ we consider the structure form $\omega_G$ defined by $(\omega_G(X_g))_h ={\rm d}L_{hg^{-1}}(X_g)$, which is a~parallelism form whose parallel fields are the left invariant vector fields. But it is also endowed with the reciprocal structure form $\omega_G^{\rm rec}$ defined by $\bigl(\omega_G^{\rm rec}(X_g)\bigr)_h = - {\rm d}L_h{\rm d}R_g^{-1}(X_g)$ whose parallel vector fields are that of $\mathfrak g^{\rm rec}$, the Lie algebra of right invariant vector fields in $G$. We have the following:
\begin{enumerate}\itemsep=0pt
\item[(a)] $t\colon \mathcal G_{x\bullet} \to M$ is a~principal bundle with structure group $\mathcal G_{xx}$.
\item[(b)] $\mathcal G_{x\bullet}$ is endowed with commuting parallelisms $j_*(\omega_{G})$ and $j_*\bigl(\omega_G^{\rm rec}\bigr)$.
\item[(c)] The Lie algebra of vector fields $j_*\bigl(\mathfrak g^{\rm rec}\bigr)$ is projectable by $t$.
\item[(d)] The projection $t_*\bigl(j_*\bigl(\mathfrak g^{\rm rec}\bigr)\bigr)$ is the space of solutions of $\mathcal D$-Lie algebra of $\mathcal G$. Thus, the differential Galois group of ${\rm Lie}(\mathcal G)$ is $\{\rm id\}$.
\end{enumerate}

The action of $G$ on $M$ naturally extends to an action of $G$ on $RM$ by $g\cdot r_x = L_g \circ r_x$. If we fix a~frame $r_x$, then the orbit $G\cdot r_x = R_{\mathcal G}M$ is a~bundle of adapted frames for $\mathcal G$. It is a~principal homogeneous space for $G$ and a~principal bundle for the stabilizer subgroup of~$x$,~$\mathcal G_{xx}$. The isomorphism $i \colon G \to \mathcal R_{\mathcal G}M$, $g\mapsto L_g\circ r_x$ allows us to give $R_{\mathcal G}M$ a~pair of commuting parallelisms~$i_*(\omega_G)$ and $i^*\bigl(\omega_G^{\rm rec}\bigr)$. We have a~commutative diagram of $G$-equivariant isomorphisms,
\[
\xymatrix{
G \ar[r] \ar[dr] & R_{\mathcal G}M = G\cdot r_x \ar[d] \\
& \mathcal G_{x\bullet},&
}
\xymatrix{
g \ar[r] \ar[dr] & L_g\circ r_x = r_xy \ar[d] \\
& j_xLg = r_y\circ r_x^{-1}.&
}
\]
The Lie algebra of vectors that are parallel for $i^*\bigl(\omega_G^{\rm rec}\bigr)$, is the Lie algebra of the infinitesimal generators of the action of $G$ on $R_{\mathcal G}M$. It follows that it is the Lie algebra restrictions to $R_\mathcal GM$ of prolongations $X^{(\infty)}$ to $RM$ of vector fields $X$ that are solutions of ${\rm Lie}(\mathcal G)$.

Summarizing, \emph{if $\mathcal G\subset \operatorname{Aut}(M)$ is a~$\mathcal D$-groupoid isomorphic to a~transitive action groupoid $G\ltimes M$, then the adapted frame bundles $R_\mathcal GM$ and the fibers of the source map $\mathcal G_{x\bullet}$ are parallelized varieties isomorphic to $(G,\omega_G)$}.

\subsection{Canonical parallelism}

The main objective of this subsection is to express how some of the local features of the action groupoid example in Section~\ref{ss_actiongroupoids} are also present in the general case of transitive finite-dimensional $\mathcal D$-groupoid. Let $\mathcal G\subset\operatorname{Aut}(M)$ be a~finite-dimensional $\mathcal D$-groupoid. We take $k$ large enough such that $\mathcal G$ is determined by its $k$-th order differential invariant and $\mathcal G \simeq \mathcal G_k$. We also assume that $\mathcal G_k\subset\operatorname{Aut}(M)$ is a~smooth Lie subgroupoid.

Let us fix a~point $x$ that we will use as the origin. We also consider the bundle $t\colon \mathcal G_{x\bullet}\to M$ of elements of $\mathcal G$ with source $x$. Note that we have a~canonical isomorphism
\[\iota\colon\ {\rm Lie}(\mathcal G)_x \xrightarrow{\sim} T_{e_x}(\mathcal G_{x\bullet}), \qquad j_xX\mapsto \widetilde{X}^{(\infty)}(e_x).\]
Under this identification, a~jet $j_xX\in {\rm Lie}(\mathcal G)_x$ is identified with the value of its prolongation at the identity \smash{$\widetilde{X}^{(\infty)}(e_x)$}.

Let us assume that the domain of $X$ is an open neighborhood\footnote{In the transcendental topology.} $U_x$ of $x$, so that $X$ is an analytic vector field on $U$. The domain of $\widetilde{X}^{(\infty)}$ as an analytic vector field in $\mathcal G$ is $t^{-1}(U_x)$. Only if the target of $\sigma\in\mathcal G$ is in $U_x$ can it be composed with the flow of $X$.

Then, if $\{j_xX_1,\dots,j_xX_r\}$ form a~basis of ${\rm Lie}(\mathcal G)_x$, we can find a~common domain $U_x$ for them, so that, \smash{$\big\langle \widetilde X_1^{(\infty)}, \dots, \widetilde X_r^{(\infty)} \big\rangle$} is a~realization of ${\rm Lie}(\mathcal G)_x$ as a~Lie algebra of vector fields in~$t^{-1}(U_x)$. Those vector fields are in the kernel of $ds$, so we may restrict them to $\mathcal G_{x \bullet} \cap t^{-1}(U_x)$, obtaining a~local analytic ${\rm Lie}(\mathcal G)_x$-parallelism of $\mathcal G_{x\bullet} \cap t^{-1}(U_x)$.
\[\rho \colon\ {\rm Lie}(\mathcal G)_x \to \mathfrak X\bigl(\mathcal G_{x\bullet} \cap t^{-1}(U_x)\bigr),
 \qquad j_xX \mapsto \widetilde X^{(\infty)}\vert_{\mathcal G_{x\bullet} \cap t^{-1}(U_x)}.\]

On the other hand, note that if $j_x\sigma \in \mathcal G_{x\bullet}$, then $\sigma$ is a~biholomorphism defined in some neighborhood of $x$ that we may also refer to as $U_x$. The main point here is that for all $z\in U_x$ we have $j_z\sigma \in \mathcal G$. This implies that the left translation by $\sigma$ is not only defined in $\mathcal G_{\bullet x}$ but it extends to an open subset $t^{-1}(U_x)$
$L_\sigma \colon t^{-1}(U_x) \to \mathcal G$, $ j_z\tau \mapsto j_z(\sigma \circ \tau)$.
Next, also consider its linearization at the identity element,
${\rm d}_{e_x}L_{\sigma} \colon T_{e_x} \mathcal G \to T_{j_x\sigma} \mathcal G$.
As $L_\sigma$ maps~$\mathcal G_{x\bullet}$ onto itself, this~${\rm d}_{e_x}L_{\sigma}$ also maps~$T_{e_x}(\mathcal G_{x\bullet})$ onto $T_{j_x\sigma}(\mathcal G_{x\bullet})$. Thus, any vector ${\vec v_{e_x} \in T_{e_x}(\mathcal G_{x\bullet})}$ can be extended to a~\emph{global} left invariant vector field $\vec A$ in $\mathcal G_{x\bullet}$ defined by $\vec A(j_x\sigma) = {\rm d}L_{\sigma}(\vec v)$.

\begin{Lemma}
Let $\vec A$ be a~left invariant vector field in $\mathcal G_{x\bullet}$, satisfying $\vec A(\sigma) = {\rm d}L_{\sigma}(\vec A(e_{s(\sigma)}))$, and~$X$ be a~vector field on some open subset $U$ {\rm(}in the transcendental topology$)$ such that for all~$z$ in its domain of definition $j_zX \in {\rm Lie}(\mathcal G)$. Then,
\smash{$\big[\vec A, \widetilde X^{(\infty)}|_{\mathcal G_{x\bullet}}\big] = 0$}
on their common domain of definition $\mathcal G_{x\bullet}\cap t^{-1}(U)$.
\end{Lemma}

\begin{proof}
The flow of $\widetilde X^{(\infty)}$ in some neighborhood of $e_x$ is given by a~left translation
$\exp\bigl(t\widetilde X^{\infty}\bigr)(\sigma) \allowbreak= \exp(t X)\circ \sigma$.
Then, we have
\[{\rm d} L_{\exp(t\widetilde X^{(\infty)})}(\vec A_{\sigma}) = \vec A_{\exp(t\vec X^{(\infty)})\circ \sigma}.\]
The vector field $\vec A$ is invariant by the flow of $\vec X^{(\infty)}$, so we conclude that they commute.
\end{proof}

\begin{Corollary}
The form $\theta \colon T\mathcal G_{x\bullet} \to {\rm Lie}(\mathcal G)_{x}^{\rm rec}$ defined by
$\theta(\vec v_{\sigma}) = \iota^{-1}\bigl({\rm d}L_{\sigma}^{-1}(\vec v_\sigma)\bigr)$ is a~regular parallelism form in $\mathcal G_{x\bullet}$ of type ${\rm Lie}(\mathcal G)_{x}^{\rm rec}$.
\end{Corollary}

\begin{proof}
For each $j_x X\in {\rm Lie}(\mathcal G)_x$, there is a~neighborhood of the identity $e_{x}$ in the transcendental topology on $\mathcal G_{x\bullet}$ in which $\widetilde X^{(\infty)}$ is defined. As ${\rm Lie}(\mathcal G)_x$ is finite-dimensional, we can extend the prolongations to $\mathcal G_{x\bullet}$ of a~basis of
${\rm Lie}(\mathcal G)_x$ defined in some neighborhood of $e_x$. Then, we have that the ${\rm Lie}(\mathcal G)_x$-valued $1$-form $\vartheta$ defined by $\vartheta\bigl(\widetilde X^{(\infty)}_{j_z\sigma}\bigr) = X_x$ is a~local parallelism form. Thus, $\theta$~is just the reciprocal form of symmetries of $\vartheta$, that happens to extend globally.
\end{proof}

Then, we have the following situation:
\begin{itemize}\itemsep=0pt
\item[(a)] $\mathcal G_{x\bullet}$ is a~parallelized variety with parallelism form $\omega$.
\item[(b)] Vector fields in $\mathcal G_{x\bullet}$ commuting with the $\omega$-parallel vector fields are the prolongations to $\mathcal G_{x\bullet}$ of vector fields in $M$ that are in the $\mathcal D$-Lie algebra of $\mathcal G$.
\end{itemize}

We can finally transport everything to the bundle of adapted frames. Let us fix a~frame $r_x\in RM$. We have now an isomorphim,
$i\colon \mathcal G_{x\bullet} \xrightarrow{\sim} R_{\mathcal G}M$, $ \sigma \mapsto \sigma\circ r_x$.
Then $\omega = i_*(\theta)$ is a~${\rm Lie}(\mathcal G)_x^{\rm rec}$-parallelism form on $R_{\mathcal G}(M)$.

\begin{Proposition}
 Let us consider $\omega$ the ${\rm Lie}(\mathcal G)_x^{\rm rec}$-parallelism form on $R_{\mathcal G}(M)$, and
let $\nabla$ be the induced Lie connection whose horizontal sections are vector fields commuting with $\omega$-parallel vector fields. Then, a~local vector field $Y$ in $R_\mathcal GM$ is $\nabla$-horizontal if and only if it is of the form~$X^{(\infty)}|_{R_\mathcal GM}$ for some local vector field $X$ in $M$ which is a~solution of ${\rm Lie}(\mathcal G)$.
\end{Proposition}

\begin{proof}
It is enough to note that the prolongation to the jet space $\vec X\mapsto \widetilde X^{(\infty)}$ commutes with the prolongation to the frame bundle, in the sense that, $i_*\bigl(\widetilde X^{(\infty)}|_{\mathcal G_{x\bullet}}\bigr) = X^{(\infty)}|_{R_\mathcal GM}$ for all vector fields in~$M$.
\end{proof}

\begin{Remark}
The form $\omega$ in $R_{\mathcal G}M$ is unique in the sense that the choice of any origin frame yields a~form with values in an isomorphic Lie algebra and the same parallel vector fields.
\end{Remark}

There is an important and implicit relation between the Galois groups of the connections $\nabla$ in $TR_\mathcal GM\to R_\mathcal GM$ and $\nabla^{{\rm Lie}(\mathcal G)}$ in ${\rm Lie}(\mathcal G)\to M$.

\begin{Proposition}\label{prp_Gal_Gal}
Let us consider $\nabla$ the Lie connection on $TR_\mathcal GM\to R_\mathcal GM$ whose horizontal vector fields are the infinitesimal symmetries of the canonical parallelism in $R_\mathcal GM$, and let $\nabla^{{\rm Lie}(\mathcal G)}$ be the connection in ${\rm Lie}(\mathcal G)\to M$ whose horizontal sections are the jet prolongations of vector fields in the $\mathcal D$-Lie algebra of $\mathcal G$. Then,
\begin{enumerate}\itemsep=0pt
 \item[$(a)$] ${\rm Gal}(\nabla) = {\rm id} \, \Longleftrightarrow \, {\rm Gal}\bigl(\nabla^{{\rm Lie}(\mathcal G)}\bigr) = \{\rm id\}$.
 \item[$(b)$] $\mathfrak{gal}(\nabla) = \{0\} \, \Longleftrightarrow \, \mathfrak{gal}\bigl(\nabla^{{\rm Lie}(\mathcal G)}\bigr) = \{\rm 0\}$.
\end{enumerate}
\end{Proposition}

\begin{proof}
Let us take $x\in M$ a~regular point of the bundles ${\rm Lie}(\mathcal G)\to M$ and $R_{\mathcal G}M\to M$, and let~${r_x \in R_{\mathcal G}M}$ be a~frame on $x$. For a~given element $j_xX\in {\rm Lie}(\mathcal G)_x$, we have that $X^{(\infty)}|_{R_\mathcal GM}$ and the section $z\to j_zX$, defined in adequate neighborhoods of $r_x$ and $x$, are horizontal sections of $\nabla$ and $\nabla^{{\rm Lie}(\mathcal G)}$, respectively. Note that $X^{\infty}|_{R_{\mathcal G}M}$ is a~rational (algebraic) vector field in~${R_\mathcal GM}$ if and only if $X$ is a~rational (algebraic) vector field in $M$, if and only if $z\to j_zX$ is a~rational section of~${J(TM/M)\to M}$. Statement (a) follows from the fact that Galois group of a~linear connection is the identity if and only if all its horizontal sections are rational. Similarly, statement (b) follows from the fact that the lie algebra of the Galois group of a~linear connection vanish if and only if all its horizontal sections are algebraic.
\end{proof}

\begin{Proposition}
Assume that $\mathcal G \simeq G\ltimes M$ is isomorphic to the action groupoid of a~faithful transitive algebraic group action. The following statements hold:
\begin{enumerate}\itemsep=0pt
\item[$(a)$] $\rm Gal\bigl(\nabla^{{\rm Lie}(\mathcal G)}\bigr) = \{\rm id\}$.
\item[$(b)$] $R_{\mathcal G}M\to M$ be a~bundle of adapted frames for $\mathcal G$ and let $\nabla$ be the induced connection in $R_{\mathcal G}M$ whose horizontal vector fields are the prolongations $X^{(\infty)}|_{R_\mathcal GM}$ of infinitesimal generators of $\mathcal G$. Then ${\rm Gal}(\nabla) = \{\rm id\}$.
\end{enumerate}
\end{Proposition}

\begin{proof}
In the case of the action groupoid $\mathcal G \simeq G\ltimes M$, we have that the vector fields in the $\mathcal D$-algebra of $\mathcal G$ are the infinitesimal generators of the action of $G$ in $M$. Its jets prolongations form a~$r$-dimensional space of regular (indeed rational) sections of $\nabla^{\rm Lie (\mathcal G)}$, and thus ${\rm Gal}\bigl(\nabla^{{\rm Lie}(\mathcal G)}\bigr) = \{{\rm id}\}$. Statement (b) follows by application of Proposition~\ref{prp_Gal_Gal}\,(a).
\end{proof}

Then, we can state our main result, that characterizes the cases in which a~$\mathcal D$-groupoid of finite dimension with centerless Lie algebra is the pullback of the action groupoid of an homogeneous space (see Theorem~\ref{thm:1}) or at least isogenous to an action groupoid (see Theorem~\ref{thm:2}).

\begin{Theorem}\label{thm:1}
Assume that $\mathcal G\subset\operatorname{Aut}(M)$ is a~transitive finite-dimensional $\mathcal D$-groupoid with centerless Lie algebra $\mathfrak g$. The following statements are equivalent:
\begin{enumerate}\itemsep=0pt
\item[$(a)$] Let \smash{$\nabla^{{\rm Lie}(\mathcal G)}$} be the connection on ${\rm Lie}(\mathcal G)\to M$ whose horizontal sections are the jet prolongations of infinitesimal generators of $\mathcal G$. Then $\rm Gal\bigl(\nabla^{{\rm Lie}(\mathcal G)}\bigr) = \{\rm id\}$.
\item[$(b)$] $R_{\mathcal G}M\to M$ be a~bundle of adapted frames for $\mathcal G$ and let $\nabla$ be the induced connection in $P$ whose horizontal vector fields are the liftings to $R_{\mathcal G}M$ of the infinitesimal generators of $\mathcal G$. Then ${\rm Gal}(\nabla) = \{\rm id\}$.
\item[$(c)$] There is an algebraic group $G$ with Lie algebra $\mathfrak g$, a~subgroup $H$ with $G$ acting faithfully on $G/H$, and a~rational dominant map $\pi\colon M\dasharrow G/H$ such that $\mathcal G$ is the pullback by $\pi$ of the action groupoid $G\ltimes(G/H)$ acting on $G/H$.
\end{enumerate}
\end{Theorem}

\begin{proof}
The equivalence between statements (a) and (b) is Proposition~\ref{prp_Gal_Gal}\,(a). Let us assume~(c). Let $\mathfrak g \subset \mathfrak X(G/H)$ be the Lie algebra of infinitesimal generators of the action of $G$ on~$G/H$. Each vector field $\vec A\in \mathfrak g$ lifts to a~rational vector field $\pi^*(\vec A)$ in $M$. Thus $\pi^{*}(\mathfrak g)$ is the Lie algebra of infinitesimal generators of the action of $\mathcal G$ on $M$ and all its members are rational vector fields in $M$. Then, by Proposition~\ref{prp_Gal_Gal}, we have (a) and (b).

Let us assume (b). Then $(R_{\mathcal G}M,\omega)$ is a~parallelized variety, where $\omega$ is the ${\rm Lie}(\mathcal G)_x^{\rm rec}$ parallelism form whose symmetries are the prolongations of infinitesimal generators of $\mathcal G$. Precisely, the Lie connection of its reciprocal parallelism is $\nabla$, and then by Theorem~\ref{thm:0} we have that there is an algebraic group $G$ with ${\rm Lie}(G)\simeq {\rm Lie}(\mathcal G)_x$ and a~rational dominant map~${\pi_0\colon R_\mathcal GM\dasharrow G}$ such that $\pi_0^*(\omega_G) = \omega$, where $\omega_G$ is its left invariant structure form. If necessary, by composing with a~left translation in $G$, we can assume that $e_G$ in the image of $\pi$ and take $r_x\in R_\mathcal GM$ such that $\pi_0(r_x) = e_G$. Let us also consider the morphism,
$\varphi \colon \mathcal G_{x\bullet}\to R_\mathcal GM$, $ \sigma_{xy} \mapsto \sigma_{xy}\circ r_x$,
and~${\widetilde\pi = \varphi\circ \pi_0}$, so that $\widetilde\pi$ is also a~morphism of parallelized varieties between~${\bigl(\mathcal G_{x\bullet },\varphi^*\omega\bigr)}$ and~${(G,\omega_G)}$. Then $\widetilde\pi|_{\mathcal G_{xx}}$ is an isogeny between the algebraic group $\mathcal G_{xx}$ and some algebraic subgroup $H\subset G$. We then have that
for any $\widetilde \pi(\sigma_{xx}\circ \tau_{xy}) = \widetilde\pi(\sigma_{xx}) \widetilde\pi(\tau_{xy})$. Then we have a~commutative diagram,
\[
\xymatrix{
\mathcal G_{x\bullet} \ar[r]\ar[d] & G \ar[d]\\
\mathcal G_{x\bullet}/\mathcal G_{xx} = M \ar[r]^-{\pi} & G/H.
}
\]
Finally, the quotient map $\pi$ is the dominant map required in statement (c).
\end{proof}

\begin{Theorem}\label{thm:2}
Assume that $\mathcal G\subset\operatorname{Aut}(M)$ is a~transitive finite-dimensional $\mathcal D$-groupoid with centerless Lie algebra $\mathfrak g$. The following statements are equivalent:
\begin{enumerate}\itemsep=0pt
\item[$(a)$] Let \smash{$\nabla^{{\rm Lie}(\mathcal G)}$} be the connection on ${\rm Lie}(\mathcal G)\to M$ whose horizontal sections are the infinitesimal generators of $\mathcal G$. Then \smash{${\mathfrak{gal}}\bigl(\nabla^{{\rm Lie}(\mathcal G)}\bigr) = \{0\}$}.
\item[$(b)$] $R_{\mathcal G}M\to M$ be a~bundle of adapted frames for $\mathcal G$ and let $\nabla$ be the induced connection in~$P$ whose horizontal vector fields are the liftings of the infinitesimal generators of $\mathcal G$. Then~${\mathfrak {gal}(\nabla) = \{0\}}$.
\item[$(c)$] There is an algebraic group $G$ with Lie algebra $\mathfrak g$, a~subgroup $H$ with $G$ acting faithfully on $G/H$, such that $\mathcal G$ is isogenous to the action groupoid $G\ltimes G/H$.
\end{enumerate}
\end{Theorem}

\begin{proof}
The equivalence between statements (a) and (b) is Proposition~\ref{prp_Gal_Gal}\,(b). Let us assume~(c). In such case, there is an algebraic variety \smash{$\widetilde M$} of the same dimension as $M$, and dominant maps \smash{$\pi_1\colon \widetilde M \to M$} and \smash{$\pi_2\colon \widetilde M \to G/H$} such that
$\pi_1^*(\mathcal G) = \pi_1^*(G\ltimes (G/H))$. We have that $\pi_2$ induces an isomorphism between the Lie algebra of the infinitesimal generators of the action of $G$ on $G/H$ and the Lie algebra of infinitesimal generators of $\pi_1^*(\mathcal G)$ in \smash{$\widetilde M$}. Then, we have that the Lie algebra of $\pi_1^*(\mathcal G)$ on \smash{$\widetilde M$} is spanned by rational vector fields. But rational vector fields on \smash{$\widetilde M$} are algebraic multivalued vector fields in $M$, so that $\mathfrak{gal}\bigl(\nabla^{{\rm Lie}(\mathcal G)}\bigr) = \{0\}$, which is statement (a).

Let us assume (a). In such case, we have a~basis of the Lie algebra of infinitesimal generators of $\mathcal G$ given by algebraic multivalued vector fields in $M$. But those algebraic multivalued vector fields are rational vector fields in a~suitable covering $\pi\colon \widetilde M \to U \subseteq M$ of a~Zariski dense open subset $U$ of $M$.
Let $\widetilde{\mathcal G} = \pi^*(\mathcal G)$; the infinitesimal generators of $\widetilde{\mathcal G}$ are then rational vector fields on \smash{$\widetilde M$}. Then, we can apply Theorem~\ref{thm:1} and obtain a~dominant map \smash{$\widetilde\pi \colon \widetilde M \to G/H$} such that~${\widetilde\pi^*(G\ltimes (G/H)) = \widetilde{\mathcal G}}$. Hence, $\mathcal G$ is isogenous to the action groupoid as required.
\end{proof}

\begin{Remark}
 The equivalences established in Theorems~\ref{thm:1} and~\ref{thm:2} resonate with the integrability results in the $\mathscr C^{\infty}$ category for Cartan connections on Lie groupoids. In~\cite[Corollary~3.12]{crainic2003integrability}, some topological conditions are given that ensure the equivalence between a~Cartan connection on a~Lie groupoid and an action groupoid. In that context, the analogous of the Galois group is the restricted holonomy of the Cartan connection in the Lie algebroid, and the required topological conditions are such that this holonomy is trivial.
\end{Remark}

\begin{Example}
Let us consider the Kummer groupoid in $\operatorname{Aut}(\mathbb C)$ given by equation~\eqref{eq_kummer}. Its main differential invariant is given in equation~\eqref{eq_invariant}. Thus, a~bundle of adapted frames $R_{\mathcal G}\mathbb C$, is defined by the differential equation,
\smash{$
\lambda_{\varepsilon\varepsilon\varepsilon} = \frac{3}{2} \frac{\lambda^2_{\varepsilon\varepsilon}}{\lambda_\varepsilon} - R(\lambda)\lambda_\varepsilon^3
$}
and it is coordinated by $\lambda$, $\lambda_{\varepsilon}$, $\lambda_{\varepsilon\varepsilon}$. Let us fix a~frame $r_0 = (\lambda_0,1,0, \lambda_{\varepsilon\varepsilon\varepsilon},\dots)$, of coordinates $\lambda(r_0) = \lambda_0$, $\lambda_\varepsilon(r_0) = 1$, $\lambda_{\varepsilon\varepsilon}(r_0) = 0$. We take an element of $\mathcal G$, $\sigma = (\lambda_0, \varphi, \varphi_\lambda, \varphi_{\lambda\lambda}, \varphi_{\lambda\lambda\lambda}, \dots)$, with source point $\lambda_0$. In the system of coordinates $\lambda$, $\lambda_\varepsilon$, $\lambda_{\varepsilon\varepsilon}$ the left translation by $\sigma$ is given by
\[L_{\sigma} \colon\
\begin{bmatrix}
\lambda \\ \lambda_\varepsilon \\ \lambda_{\varepsilon\varepsilon}
\end{bmatrix}
\mapsto
\begin{bmatrix}
\sigma(\lambda) \\ \sigma'(\lambda)\lambda_\varepsilon \\ \sigma''(\lambda)\lambda_\varepsilon^2 + \sigma'(\lambda)\lambda_{\varepsilon\varepsilon}
\end{bmatrix}; \qquad
\begin{bmatrix}
\lambda_0 \\ 1 \\ 0
\end{bmatrix}
\mapsto
\begin{bmatrix}
\varphi \\ \varphi_\lambda \\ \varphi_{\lambda\lambda}
\end{bmatrix} =
\begin{bmatrix}
\lambda \\ \lambda_\varepsilon \\ \lambda_{\varepsilon\varepsilon}
\end{bmatrix}.
\]
Note that when $\sigma$ runs along $\mathcal G_{\lambda_0\bullet}$ then $L_\sigma(r_0)$ runs along $R_\mathcal G\mathbb C$. The Jacobian matrix of $L_{\sigma}$ at~$r_0$ is computed directly, so that we obtain
\[[d_{r_0}L_\sigma] =
\begin{bmatrix}
 \varphi_\lambda & 0 & 0 \\
 \varphi_{\lambda\lambda} & \varphi_\lambda & 0 \\
 \varphi_{\lambda\lambda\lambda} & 2\varphi_{\lambda\lambda} & \varphi_\lambda
\end{bmatrix}.
\]
Let us set a~basis of left invariant vector fields by propagating a~basis of $T_{r_0}(R_\mathcal G\mathbb C)$ along $R_{\mathcal G}\mathbb C$
\begin{gather*}
Y_0(\sigma\circ r_0)= {\rm d}_{r_0}L_\sigma \left(\frac{\partial}{\partial \lambda} \right),\qquad
Y_1(\sigma \circ r_0)= {\rm d}_{r_0}L_\sigma \left(\frac{\partial}{\partial \lambda_{\varepsilon}} \right),\\
Y_2(\sigma \circ r_0)= {\rm d}_{r_0}L_\sigma \left(\frac{\partial}{\partial \lambda_{\varepsilon\varepsilon}} \right).
\end{gather*}
Then,
$Y_0 = \lambda_\varepsilon\frac{\partial}{\partial \lambda} + \lambda_{\varepsilon\varepsilon}\frac{\partial}{\partial \lambda_\varepsilon} + \varphi_{\lambda\lambda\lambda}\frac{\partial}{\partial \lambda_{\varepsilon\varepsilon}}$.
In order to express $\varphi_{\lambda\lambda\lambda}$ in terms of the coordinates in~$R_\mathcal GM$, we need to use the groupoid equation~\eqref{eq_kummer}, obtaining
\[\varphi_{\lambda\lambda\lambda} = R(\lambda_0)\lambda_{\varepsilon} + \frac{3}{2}\frac{\lambda_{\varepsilon\varepsilon}^2}{\lambda_\varepsilon} - R(\lambda)\lambda_\varepsilon^3.
\]
And then we have the basis of parallel vector field
\begin{gather*}
Y_0 = \lambda_\varepsilon\frac{\partial}{\partial \lambda} + \lambda_{\varepsilon\varepsilon}\frac{\partial}{\partial \lambda_\varepsilon} +\left(
 R(\lambda_0)\lambda_{\varepsilon} + \frac{3}{2}\frac{\lambda_{\varepsilon\varepsilon}^2}{\lambda_\varepsilon} - R(\lambda)\lambda_\varepsilon^3
\right)
\frac{\partial}{\partial \lambda_{\varepsilon\varepsilon}},\\
Y_1 = \lambda_\varepsilon\frac{\partial}{\partial \lambda_\varepsilon} + 2\lambda_{\varepsilon\varepsilon}\frac{\partial}{\partial \lambda_{\varepsilon\varepsilon}}, \qquad Y_2 = \lambda_\varepsilon\frac{\partial}{\partial \lambda_{\varepsilon\varepsilon}}.
\end{gather*}
Let us compare this basis with the $\mathfrak{sl}_2$-parallelism of $R_\mathcal G\mathbb C$ given in~\cite[Section 3.5.2]{bc2017} with parallel vector fields
\begin{gather*}
E_{-1} = \lambda_{\varepsilon} \frac{\partial }{\partial \lambda} +
\lambda_{\varepsilon\varepsilon} \frac{\partial }{\partial \lambda_\varepsilon} +
\left(\frac{3}{2}\frac{\lambda_{\varepsilon\varepsilon}^2}{\lambda_{\varepsilon}} - R(\lambda)\lambda_{\varepsilon}^3 \right)\frac{\partial}{\partial \lambda_{\varepsilon\varepsilon}},\\
E_0 = \lambda_{\varepsilon}\frac{\partial}{\partial \lambda_{\varepsilon}} + 2 \lambda_{\varepsilon\varepsilon} \frac{\partial}{\partial \lambda_{\varepsilon\varepsilon}}, \qquad
E_1 = 2 \lambda_{\varepsilon}\frac{\partial}{\partial \lambda_{\varepsilon\varepsilon}}.
\end{gather*}
It is clear that
\[
\begin{bmatrix}
 Y_0 \\ Y_1 \\ Y_2
\end{bmatrix} =
\begin{bmatrix}
 1 & 0 & \frac{R(\lambda_0)}{2} \\
 0 & 1 & 0 \\
 0 & 0 & 2
\end{bmatrix}
\begin{bmatrix}
 E_{-1} \\ E_0 \\ E_1
\end{bmatrix}.
\]
Therefore, $Y_0$, $Y_1$, $Y_2$ are a~basis of parallel vector fields for an $\mathfrak{sl}_2$-parallelism in $R_\mathcal GM$, as expected. Now, let us see that its reciprocal parallelism is that given by the prolongation of vector fields in the $\mathcal D$-Lie algebra of $\mathcal G$. Let us consider a~vector field $X = a(\lambda)\frac{\partial}{\partial \lambda}$ and its prolongation of~$X^{(2)}$ to $R_\mathcal GM$,
\[X^{(2)} = a\frac{\partial}{\partial \lambda} + a_\lambda\lambda_\varepsilon\frac{\partial}{\partial \lambda_\varepsilon}+
\bigl(a_{\lambda\lambda} \lambda_\varepsilon^2 + a_\lambda \lambda_{\varepsilon\varepsilon}\bigr)\frac{\partial}{\partial \lambda_{\varepsilon\varepsilon}}.\]
We now compute the Lie brackets of $X^{(2)}$ with the basis of parallel vector fields $Y_0$, $Y_1$, $Y_2$. A~direct algebraic computation yields
\[
\big[Y_0, X^{(2)}\big] = (a_{\lambda\lambda\lambda} + 2R(\lambda)a_\lambda + R'(\lambda)a)\lambda_\varepsilon^3 \frac{\partial}{\partial \lambda_{\varepsilon\varepsilon}}, \qquad \big[Y_1,X^{(2)}\big] = \big[Y_2,X^{(2)}\big] = 0.\]
It is then clear that $X^{(2)}$ commutes with $Y_0$, $Y_1$, $Y_2$ if and only if $a$ satisfies~\eqref{eq_sym2}, or equivalently, $X$ is in the $\mathcal D$-Lie algebra of $\mathcal G$.
\end{Example}

\subsection[Minimality and n-minimality]{Minimality and $\boldsymbol{n}$-minimality}

There is a~natural embedding of $\operatorname{Aut}(M)\times\operatorname{Aut}(M) \subset \operatorname{Aut}(M\times M)$,
$\operatorname{Aut}(M)\times \operatorname{Aut}(M)$ is the first-order $\mathcal D$-groupoid on $M\times M$ defined by equations,
$\frac{\partial \sigma_1}{\partial x_2}= \frac{\partial \sigma_2}{\partial x_1} = 0$
for $\sigma = (\sigma_1,\sigma_2) \colon M\times M \to M\times M$, $(x_1,x_2) \mapsto (\sigma_1(x_1,x_2),\sigma_2(x_1,x_2))$. However, it was first noted by Sophus Lie that the diagonal action of $\operatorname{Aut}(M)$ on $M$, that is, the equation~${\sigma_1 = \sigma_2}$ can not be defined by using differential equations. There is a~diagonal pseudogroup of $\operatorname{Aut}(M)\times \operatorname{Aut}(M)$ acting on $M\times M$, but it is not a~$\mathcal D$-groupoid (or Lie pseudogroup). In broad terms, $\mathcal D$-groupoids may admit fewer sub-$\mathcal D$-groupoids than algebraic groups admit algebraic subgroups.

If $\mathcal G$ is isogenous to an action groupoid $G\ltimes M$, then clearly the diagonal of $\mathcal G\times \mathcal G$ is a~sub-$\mathcal D$-groupoid isogenous to the diagonal of $G$ acting on $M\times M$. This motivates the following definition.

\begin{Definition}\label{def_minimal}
 A~$\mathcal D$-groupoid $\mathcal G$ on $M$ is $n$-minimal if for any sub-$\mathcal D$-groupoid $\mathcal H \subset \mathcal G^n$, if~${\dim_{M^n} (\mathcal H)< \dim_{M^n}\bigl(\mathcal G^n\bigr)}$, then there exists $i$ such that $\dim_M({\rm pr}_i(\mathcal H)) =0$, where ${\rm pr}_i \colon \mathcal G^n \to \mathcal G$ is the projection on the $i$-th factor.
\end{Definition}

\begin{Remark}
 We will say minimal for $1$-minimal. That is, $\mathcal G$ is minimal if it does not admit nontrivial sub-$\mathcal D$-groupoids of positive dimension.
\end{Remark}

\begin{Example} First examples of minimal $\mathcal D$-groupoid were given in~\cite[Proposition 3.6]{Casale_BlazquezSanz_ArenasTirado_2025}, where if is shown that the Kummer groupoid of symmetries of a~projective structure~\eqref{eq_kummer} is minimal if and only if the Ricatti equation
\begin{equation}\label{eq_ric}
u_\lambda + u^2 + \frac{1}{2}R(\lambda)=0
\end{equation}
has no solutions in $\mathbb C(\lambda)^{\rm alg}$, the algebraic closure of $\mathbb C(\lambda)$. This relation between \eqref{eq_ric} and the third-order linear equation~\eqref{eq_sym2} comes through the
second-order differential equation
\begin{equation}\label{linear}
\psi_{\lambda\lambda} + \frac{1}{2}R(\lambda)\psi = 0.
\end{equation}
On one hand, logarithmic derivatives $u = \psi_\lambda/\psi$ of solutions of~\eqref{linear} satisfy the differential equation~\eqref{eq_ric}. On the other hand, \eqref{eq_sym2} is the second symmetric power of~\eqref{linear} in the sense that for any basis of solutions $\psi_1$, $\psi_2$ of~\eqref{linear} the quadratic products
$
a_1 = \psi_1^2$, $ a_2 = \psi_1\psi_2$, $ a_3 = \psi_2^2$
form a~basis of solutions of~\eqref{eq_sym2}.

Let us consider $x\in \mathbb C$ a~regular point for $R$,
the Galois group ${\rm Gal}\bigl(\nabla^{{\rm Lie}(\mathcal G)}, x\bigr)$ acts on the space ${\rm Lie}(\mathcal G)_{x}\simeq \mathfrak{sl}_2$ by automorphisms of Lie algebras. Then we have
\[
{\rm Gal}\bigl(\nabla^{{\rm Lie}(\mathcal G)}, x\bigr) \subset \operatorname{Aut}({\rm Lie}(\mathcal G)_{x}) \simeq {\rm PSL}_2.\]
On one hand, any nontrivial positive-dimensional sub-$\mathcal D$-groupoid of $\mathcal G$ has a~nontrivial sub-$\mathcal D$-algebra of ${\rm Lie}(\mathcal G)$. Such a~nontrivial subalgebra of ${\rm Lie}(\mathcal G)$ will give us a~nontrivial ${\rm Gal}\bigl(\nabla^{{\rm Lie}(\mathcal G)}, x\bigr)$-invariant Lie subalgebra of ${\rm Lie}(\mathcal G)_{x}$, and then the reduction of the action of the Galois group on the space of initial conditions.

Reciprocally, the realization of~\eqref{eq_sym2} as the symmetric square of~\eqref{linear} informs us that any ${\rm Gal}\bigl(\nabla^{{\rm Lie}(\mathcal G)}, x\bigr)$-invariant subspace of ${\rm Lie}(\mathcal G)_x$ also implies that the Galois group of~\eqref{linear} is inside a~Borel subgroup of ${\rm SL}_2$. In such case, \eqref{eq_ric} has a~rational solution.

This rational solution $u$ can be used to find a~proper sub-$\mathcal D$-groupoid in the following way. Take $r = -2u \to r' = \frac{1}{2}r^2+R$, and define
\[\mathcal{B} = \{j_x (a \partial_\lambda) \in J(T\mathbb C/\mathbb C) \mid a_{\lambda\lambda}+ra_\lambda+r_\lambda a=0\}.
\]
Then $\mathcal{B}$ is a~2-dimensional subalgebra of ${\rm Lie}(\mathcal G)$ as it follows from $R = r'-\frac{1}{2}r^2$. It is a~Lie $\mathcal D$-algebra of affine symmetries of dimension 2, sitting inside the $3$-dimensional projective symmetry algebra ${\rm Lie}(\mathcal G)$.

Consider the transformation rule $r^*=\frac{\varphi_{\lambda\lambda}}{\varphi_\lambda}+\varphi_\lambda r(\varphi)$.
Differentiate at $t = 0$ along a~flow $\varphi_t$ with generator $a(\lambda)\partial_\lambda$
\[\delta r := \dfrac{\rm d}{{\rm d}t}\Big|_{t=0}r^* = a_{\lambda\lambda}+ra_\lambda+r_\lambda a,\]
so $a \in \mathcal{B} \Longleftrightarrow \delta r = 0 \Longleftrightarrow \text{the flow preserves the affine connection}\ r$ and the set of diffeomorphisms preserving $r$ is automatically a~nontrivial $2$-dimensional sub-pseudogroup of $\mathcal G$. Then we have the equivalence between the following conditions:
\begin{itemize}\itemsep=0pt
 \item[(i)] The associated Riccati equation~\eqref{eq_ric}
 has no rational solution.
 \item[(ii)] There is no nontrivial ${\rm Gal}\bigl(\nabla^{{\rm Lie}(\mathcal G)},x\bigr)$-invariant subspace in ${\rm Lie}(\mathcal G)_x$.
 \item[(iii)] There is no nontrivial ${\rm Gal}\bigl(\nabla^{{\rm Lie}(\mathcal G)},x\bigr)$-invariant Lie subalgebra in ${\rm Lie}(\mathcal G)_x$.
 \item[(iv)] $\mathcal G$ is minimal.
\end{itemize}
\end{Example}

\begin{Remark}
The reader may note that in~\cite[Proposition 3.6]{Casale_BlazquezSanz_ArenasTirado_2025} it is stated that the existence of an algebraic but not rational solution of the Ricatti equation~\eqref{eq_ric} also implies the existence of a~proper sub-$\mathcal D$-groupoid of $\mathcal G$. However, such a~$\mathcal D$-groupoid is not a~$\mathcal D$-groupoid $\mathcal H\subset\mathcal G\subset \operatorname{Aut}(\mathbb C)$. Such a~groupoid is of the form $\mathcal H\subset \pi^{*}(\mathcal G) \subset \operatorname{Aut}(M)$ where $\pi\colon M\to \mathbb C$ is the algebraic curve associated to the algebraic solution of~\eqref{eq_ric}. Therefore, it is not considered here as a~proper sub-$\mathcal D$-groupoid of $\mathcal G$.
\end{Remark}

\begin{Definition}
For a~Lie algebra $\mathfrak g$, a~subgroup of $\operatorname{Aut}(\mathfrak g)$ is Lie-irreducible if there is no invariant proper Lie subalgebra.
\end{Definition}

\begin{Theorem}\label{thm:6}
 Let $\mathcal G$ be a~finite-dimensional transitive $\mathcal D$-groupoid on $M$, let $x\in M$ be a~regular point for $\mathcal G$, and let ${\rm Lie}(\mathcal G)_x$ be the Lie algebra at $x$.
 If ${\rm Gal}({\rm Lie}(\mathcal G),x) \subset \operatorname{Aut}({\rm Lie}(\mathcal G)_x)$ is Lie-irreducible, then for any $n \in \mathbb N$, $\mathcal G$ is $n$-minimal.
\end{Theorem}

\begin{proof}
Let us first prove the minimality (= $1$-minimality). If $\mathcal H \subsetneq \mathcal G$ is such that $\dim_M \mathcal H < \dim_M \mathcal G$, then ${\rm Lie}(\mathcal H)\subsetneq {\rm Lie}(\mathcal G)$ is a~vector sub-bundle invariant under the connection. Then the fiber at $x$: ${\rm Lie}(\mathcal H)_x \!\subsetneq\! {\rm Lie}(\mathcal G)_x$ is a~Lie subalgebra invariant under the Galois group ${\rm Gal}\bigl(\nabla^{{\rm Lie}(\mathcal G)}, x\bigr)$. Under our hypothesis, it is $\{0\}$ meaning that $\dim_M \mathcal H = 0$.

Now assume $n > 0$. Let $\mathcal H\subsetneq \mathcal G^n$ such that $\dim_{M^n} \mathcal H < \dim_{M^n} \mathcal G^n$. Since
\[
{\rm Gal}\bigl({\rm Lie}(\mathcal G) ^n,(x_1, \dots, x_n)\bigr) = \prod {\rm Gal}({\rm Lie}(\mathcal G), x_i),
\] the possible proper Lie algebras
\[
{\rm Lie}(\mathcal H)_{(x_1, \dots, x_n)} \subset \prod \mathfrak {\rm Lie}(\mathcal G)_{x_i}
\] invariant under ${\rm Gal}\bigl({\rm Lie}(\mathcal G) ^n,(x_1, \dots, x_n)\bigr)$ are the products of factors: ${\rm Lie}(\mathcal H)_{(x_1, \dots, x_n)} = \prod \mathfrak h_i$. Then there exists $i$ such that ${\rm pr}_i({\rm Lie}(\mathcal H)_{(x_1, \dots, x_n)})$ is a~proper Lie subalgebra of ${\rm Lie}(\mathcal G)_{x_i}$ invariant under ${\rm Gal}({\rm Lie}(\mathcal G), x_i)$. By Lie-irreducibility it is~$\{0\}$.

This implies that if $\mathcal H \subset \mathcal G^n$ is such that $\dim_{M^n}(\mathcal H) < \dim_{M^n}\bigl(\mathcal G^n\bigr)$ then its projection on the $i$-th factor is 0-dimensional.
\end{proof}

By simple Lie algebra, we mean a~finite-dimensional noncommutative Lie algebra without proper ideals.

\begin{Theorem}\label{thm:7}
 Let $\mathcal G$ be a~finite-dimensional transitive $\mathcal D$-groupoid on $M$, let $x\in M$ be a~regular point of $\mathcal G$, and let ${\rm Lie}(\mathcal G)_x$ be the Lie algebra at $x$. Assume ${\rm Lie}(\mathcal G)_x$ is simple and ${\rm Gal}\bigl(\nabla^{{\rm Lie}(\mathcal G)}, x\bigr) \not = \{ {\rm id} \}$.
 For any $n >1$ there exists no proper sub-$\mathcal D$-groupoid $\mathcal H \subset \mathcal G^n$ whose projections on each factor are onto.
\end{Theorem}

\begin{proof}
Let $\mathcal H \subset \mathcal G^n$ be a~proper subgroupoid with projections on each factor onto. Let ${\rm Lie}(\mathcal H)_{(x_1,\dots,x_n)} \subset \prod \mathfrak {\rm Lie}(\mathcal G)_{x_i}$ be its Lie algebra at a~generic regular point $(x_1,\dots, x_n) \in M^n$. As the projections on each factor are onto, by Goursat's lemma, there exists $\ell<k$ such that the projection onto the $(\ell,k)$-th factors of ${\rm Lie}(\mathcal H)_{(x_1,\dots,x_n)}$ is the graph of an isomorphism ${\rm Lie}(\mathcal G)_{x_\ell} \to {\rm Lie}(\mathcal G)_{x_k}$.

As ${\rm Lie}(\mathcal H) \subset {\rm Lie}(\mathcal G)^n$ is a~\smash{$\nabla^{{\rm Lie}(\mathcal G)^n}$}-invariant sub-bundle, then its fiber ${\rm Lie}(\mathcal H)_{(x_1,\dots,x_n)}$ must be invariant under the action of \smash{${\rm Gal}\bigl(\nabla^{{\rm Lie}(\mathcal G)^n}, (x_1, \dots, x_n)\bigr)$}. This implies that the projection must be ${\rm Gal}\bigl(\nabla^{{\rm Lie}(\mathcal G)},x_\ell\bigr) \times {\rm Gal}\bigl(\nabla^{{\rm Lie}(\mathcal G)}, x_k\bigr)$ invariant. This implies ${\rm Gal}\bigl(\nabla^{{\rm Lie}(\mathcal G)}, x_\ell\bigr) = \{{\rm id}\}$.
\end{proof}

\begin{Corollary}
 Let $\mathcal G$ be a~finite-dimensional transitive $\mathcal D$-groupoid on $M$, and $x\in M$ be a~regular point for $\mathcal G$, and let ${\rm Lie}(\mathcal G)_x$ be the Lie algebra at $x$. Assume ${\rm Lie}(\mathcal G)_x$ is simple and $\mathcal G$ is minimal.
 Then for all $n \geq 1$, $\mathcal G$ is $n$-minimal.
\end{Corollary}

\begin{proof}
If ${\rm Lie}(\mathcal G)_x$ is simple, then its center is trivial. If ${\rm Gal}\bigl(\nabla^{{\rm Lie}(\mathcal G)},x\bigr) = \{\rm id\}$ then, by Theorem~\ref{thm:1}, $\mathcal G$ is the pullback of an action groupoid $G\ltimes (G/H)$ by a~rational dominant map $\pi\colon M \dasharrow G/H$. Any nontrivial algebraic subgroup $K \subset G$ defines a~subgroupoid $\pi^*(K\ltimes (G/H))\subset \mathcal G$. For instance, the algebraic group generated by a~vector field is a~commutative subgroup and, by simplicity, a~proper subgroup of positive dimension. This proves that if $\mathcal G$ is minimal, then ${\rm Gal} \bigl(\nabla^{{\rm Lie}(\mathcal G)},x\bigr)\not = \{ {\rm id}\}$. Finally, Theorem~\ref{thm:7} proves the corollary.
\end{proof}

\begin{Definition}
Let $\mathcal G$ be a~$\mathcal D$-groupoid on $M$. It is said to act diagonally on $M\times M$ if there exists a~$\mathcal D$-groupoid $\mathcal H \subset \mathcal G \times \mathcal G$ such that
\begin{itemize}\itemsep=0pt
\item[(a)] $\dim_M \mathcal G = \dim_{M\times M} \mathcal H$.
\item[(b)] The diagonal map $\Delta\colon \mathcal G \to \mathcal G \times \mathcal G$ has values in $\mathcal H$.
\end{itemize}
\end{Definition}

\begin{Corollary}\label{corollary16}
Let $\mathcal G$ be a~finite-dimensional transitive $\mathcal D$-groupoid on $M$, let $x\in M$ be a~regular point for $\mathcal G$, and let ${\rm Lie}(\mathcal G)_x$ be the Lie algebra at $x$. Assume ${{\rm Lie}(\mathcal G)}_x$ is centerless.
 The $\mathcal D$-groupoid~$\mathcal G$ acts diagonally if and only if $ \mathcal G$ is algebraically integrable in the sense of Theorem~{\rm~\ref{thm:1}}, that is, $\mathcal G$ is the pullback of an action groupoid $G\ltimes (G/H)$ by a~dominant map $\pi\colon M\to G/H$.
\end{Corollary}
\begin{proof}
It is clear that algebraic group actions act diagonally and so does any $\mathcal D$-groupoid on~$M$ which is a~pullback of an action groupoid.

Now assume $\mathcal G$ acts diagonally. Let $(x_1,x_2)\in M \times M$ and ${{\rm Lie}(\mathcal H)}_{(x_1,x_2)} \subset {\rm Lie}(\mathcal G)_{x_1} \times {\rm Lie}(\mathcal G)_{x_2}$ be the Lie algebras. As $\mathcal H$ contains the diagonal embedding of $\mathcal G$, both projections of ${{\rm Lie}(\mathcal H)}_{(x_1,x_2)}$ onto ${\rm Lie}(\mathcal G)_{x_1}$ and ${\rm Lie}(\mathcal G)_{x_2}$ are onto. Then the dimension condition ensures that it is the graph of an isomorphism.

As before, this graph should be ${\rm Gal}\bigl(\nabla^{{\rm Lie}(\mathcal G)},x_1\bigr) \times {\rm Gal}\bigl(\nabla^{{\rm Lie}(\mathcal G)},x_2\bigr)$ invariant. This implies that these Galois groups are trivial. Finally, Theorem~\ref{thm:1} proves the corollary.
\end{proof}

\subsection*{Acknowledgments}

The authors would like to thank the referees for valuable comments and suggestions that improved the quality of this article.
AAT thanks to the doctoral scholarship program from the Universidad de Antioquia for the support and funding.
DBS and AAT acknowledge the support of Universidad Nacional de Colombia - Sede Medell\'in through project \emph{Teoría de Galois diferencial y modelos de sistemas dinámicos algebraicos} (HERMES 60920).
\noindent GC is partially supported by the project ANR-25-CE40-1360 ``IsoMoDyn" and by Centre Henri Lebesgue, program ANR-11-LABX-0020-0.

\pdfbookmark[1]{References}{ref}
\LastPageEnding

\end{document}